\documentclass{amsart}

\usepackage{amsmath}%
\usepackage{amsfonts}%
\usepackage{amssymb}%
\usepackage{cite}
\usepackage{enumerate}
\usepackage{graphicx}
\usepackage{verbatim}
\usepackage{textcomp}
\input xy
\xyoption{all}
\usepackage[pdfborder={0 0 0}]{hyperref}

\newtheorem{theorem}{Theorem}[subsection]

\theoremstyle{plain}

\newtheorem{cor}[theorem]{Corollary}
\newtheorem{prop}[theorem]{Proposition}
\newtheorem{lemma}[theorem]{Lemma}

\theoremstyle{definition}

\newtheorem{definition}[theorem]{Definition}
\newtheorem{example}[theorem]{Example}

\newtheorem*{notation}{Notation}

\newtheorem{remark}[theorem]{Remark}

\newtheorem{obs}[theorem]{Observation}

\theoremstyle{remark}
\newtheorem{case}{Case}

\numberwithin{equation}{section}

\DeclareMathOperator{\q}{\mathfrak{q}}
\DeclareMathOperator{\m}{\mathfrak{m}}
\DeclareMathOperator{\n}{\mathfrak{n}}
\DeclareMathOperator{\p}{\mathfrak{p}}
\DeclareMathOperator{\X}{\mathfrak{X}}
\DeclareMathOperator{\Y}{\mathfrak{Y}}
\DeclareMathOperator{\OO}{\mathcal{O}}
\DeclareMathOperator{\M}{\mathcal{M}}
\DeclareMathOperator{\R}{\mathcal{R}}

\DeclareMathOperator{\E}{\mathcal{E}}

\DeclareMathOperator{\G}{\mathcal{G}}
\DeclareMathOperator{\F}{\mathcal{F}}

\begin{document}
\title{Birational Spaces}
\author{Uri Brezner}
%\address[A. One and A. Two]
%{Author OneTwo common address, line 1 \newline%
%\indent Author OneTwo common address, line 2}%
\email{uri.brezner@mail.huji.ac.il}%
\date{\today}
%\subjclass{Primary 05C38, 15A15; Secondary 05A15, 15A18} %
%\keywords{Keyword one, keyword two etc.}%
%\dedicatory{Dedicated to Professor XY on the occasion of his seventieth birthday.}

\begin{abstract}
In this paper we construct the category of \emph{birational spaces} as the category in which the relative Riemann-Zariski spaces of \cite{tem} are naturally included.
Furthermore we develop an analogue of Raynaud's theory.
We prove that the category of quasi-compact and quasi-separated birational spaces is naturally equivalent to the localization of the category of pairs of quasi-compact and quasi-separated schemes with an affine schematically dominant morphism between them localized with respect to relative blow ups and relative normalizations.

\end{abstract}
\maketitle

%\tableofcontents

%\newpage

\section{Introduction}

In the 1930's and 1940's Oscar Zariski studied the problem of resolution of singularities for varieties of characteristic zero.
He introduced the notion of the \emph{Riemann-Zariski space}\footnote{Zariski originally called it the \emph{Riemann manifold} \cite{zar_uniformization}.
Later Nagata offered the name \emph{Zariski-Riemann space} \cite{nag1} to avoid confusion with the Riemann manifold of differential geometry.
In \cite{tem} Temkin calls this the \emph{Riemann-Zariski space} and we follow suite.} of a finitely generated field extension $k \subset K$, denoted $RZ_K(k)$.
This is the space of all valuations on $K/k$ of dimension zero.
Later he showed that the Riemann-Zariski space can be obtained as the projective limit of all projective models of $K \slash k$ \cite{zar_compactness}.

\begin{comment}
Nagata then used the Riemann-Zariski space in proving what would be called the \emph{Nagata's compactification theorem} \cite{nag1,nag2}.
In the language of schemes, it states that if $X$ is a Noetherian scheme and $f:Y \to X$ is separated and of finite type, then there exists a proper $X$-scheme $Z$ and an open immersion $j:Y \to Z$ such that $f$ factors as $Y \stackrel{j}{\to} Z \to X$.
\end{comment}

Temkin introduced a relative notion, the \emph{relative Riemann-Zariski space}, $RZ_Y(X)$ for a separated morphism of quasi-compact and quasi-separated schemes $f:Y \to X$.
He defined $RZ_Y(X)$ as the projective limit, of the underlying topological spaces, of all the $Y$-modifications of $X$ \cite{tem2,tem}.

\begin{comment}
Temkin used the relative Riemann-Zariski space to prove a decomposition theorem that generalizes Nagata's compactification theorem.
Temkin's theorem states that a morphism of quasi-compact and quasi-separated schemes $f:Y \to X$ is separated if and only if there exists a proper $X$-scheme $Z$ and an affine morphism $j:Y \to Z$ such that $f$ factors as $Y \stackrel{j}{\to} Z \to X$ \cite{tem}.
\end{comment}
 
Temkin showed that $RZ_Y(X)$ is isomorphic to the space consisting of  unbounded $X$-valuations on $Y$ equipped with a suitable topology.

Our first aim in this paper is to provide a categorical approach to $RZ$ spaces through the valuation point of view.
Our approach is to first define for given rings $A \to B$ an affinoid birational space $Val(B,A)$ of unbounded $A$-valuations on $B$.
Then general birational spaces $Val(Y,X)$ are glued from affinoid ones along affinoid subdomains.

We restrict our study only to the case of affine, schematically dominant morphisms $f:Y \to X$ of quasi-compact and quasi-separated schemes.
However this is essentially the same as assuming that $f:Y \to X$ is a separated morphism:
by Temkin's decomposition theorem \cite[Theorem 1.1.3]{tem} any separated morphisms $f:Y \to X$ of quasi-compact and quasi-separated schemes factors as $Y \stackrel{j}{\to} Z \to X$ where $j:Y \to Z$ is an affine, schematically dominant morphism and $Z \to X$ is proper.
It will become clear from the construction that $Val(Y,X)=Val(Y,Z')$ by the valuative criterion for properness, so our results hold for separated morphisms.

Our second aim is to develop an analogue of Raynaud's theory.
Let $R$ be a valuation ring of Krull dimension 1, complete with respect to the $J$-adic topology generated by a principal ideal $J = (\pi) \subset R$ where $\pi$ is some non-zero element of the maximal ideal of $R$, and $K$ the fraction field of $R$.
It is then possible to talk about the category of admissible formal $R$-schemes.
On the other hand it is also possible to talk about the category of rigid $K$-spaces.
It was Raynaud \cite{R} who suggested to view rigid spaces entirely within
the framework of formal schemes.
Elaborating the ideas of Raynaud, it is proved in \cite{bl1} that the category of admissible formal $R$-schemes, localized with respect to class of admissible formal blow ups, is naturally equivalent to the category of rigid $K$-spaces which are quasi-compact and quasi-separated.

We will show that the localization of the category of pairs of quasi-compact and quasi-separated schemes with an affine, schematically dominant morphism between them localized with respect to relative blow ups and relative normalizations is naturally equivalent to the category of quasi-compact and quasi-separated birational spaces.

Let $A \subset B$ be commutative rings with unit.
We define spaces of pairs of rings $Spa(B,A)$, and affinoid birational spaces $Val(B,A)$ which is our main interest in Section 2.
We study some of their topological properties and endow $Val(B,A)$ with two sheaves of rings $\OO_{Val(B,A)} \subset \M_{Val(B,A)}$ both making $Val(B,A)$ a locally ringed space.
The main highlight of Section 3 is the proof that the functor $Val$ gives rise to an anti-equivalence from the localization of the category of pairs of rings with respect to relative normalizations to the category of affinoid birational spaces.
Also in Section 3 we globalize the construction by introducing the notion of a general birational space.
These are topological spaces equipped with a pair of sheaves such that the space is locally ringed with respect to both sheaves and is locally isomorphic to $Val(A,B)$.
Finally, in Section 5 we prove that the localization of the category of
pairs of quasi-compact and quasi-separated schemes with an affine schematically dominant morphism between them localized with respect to relative blow ups and relative normalizations is naturally equivalent to the category of quasi-compact and quasi-separated birational spaces.
For the last step, Section 4 is dedicated to the further development of the theory of relative blow ups, and, in particular, prove the universal property of relative blow ups.

\section{Construction of the Space \textit{Val(B,A)}}
Throughout all rings are assumed to be commutative with unity.

\subsection{Valuations on Rings} \label{vals on rings}
In this Subsection we fix terminology and collect general known facts about valuations.

Given a totally ordered abelian group $\Gamma$ (written multiplicatively), we extend $\Gamma$ to a totally ordered monoid $\Gamma\cup\{ 0\}$ by the rules
\begin{displaymath}
0\cdot\gamma=\gamma\cdot0=0 \text{ and }  0<\gamma \qquad \forall \; \gamma \in \Gamma.
\end{displaymath} 

\begin{definition}
Let $B$ be a ring and $\Gamma$ a totally ordered group. A \emph{valuation} $v$ on $B$ is a map $ v:B \to \Gamma\cup\{ 0\}$ satisfying the conditions
\begin{itemize}
\item $v(1)=1$

\item $v(xy)=v(x)v(y) \qquad \forall{x,y}\in B$

\item $v(x+y)\le \max\{v(x),v(y)\} \qquad \forall{x,y}\in B$.

\end{itemize}
\end{definition}

\noindent Note that $\p=\ker{v}= \{ b \in B \mid v(b)=0 \} $ is a prime ideal in $B$.\\
We furthermore assume that $\Gamma$ is generated, as an abelian group, by $v(B-\,\p)$.

\begin{remark}
When $B$ is a field the above definition coincides with the classical definition of a valuation with the value group written multiplicatively.
\end{remark}

Let $v$ be a valuation on $B$ with kernel $\p$.
Denote the residue field of $\p$ by $k(\p)$.
We obtain a diagram
 \begin{displaymath}
    \xymatrix {
        B \ar [d] \ar [r]^-{v}       & \Gamma\cup\{ 0\}  \\
       {B /\p} \ar [d]              &                    \\
       k(\p) \ar [uur]^{\bar {v}}
       }
\end{displaymath} 
where $\bar{v}:k(\p)\to\Gamma\cup\{ 0\}$ is a valuation on $k(\p)$ induced by $v$.
On the other hand a prime ideal $\p \in SpecB$ and a valuation $\bar{v}$ on the residue field $k(\p)$ uniquely determine a valuation $v$ on $B$ with kernel $\p$ by setting
$$v(b)=\begin{cases}
\bar{v}(\bar{b}) & \text{if}\; b \notin \p \\
0 & \text{if}\;  b \in \p
\end{cases}$$
where $\bar{b}$ is the image of $b$ in $k(\p)$.
Hence giving a valuation $v$ on $B$ is equivalent to giving a prime ideal $\p$ and a valuation $\bar{v}$ on the residue field $k(\p)$.

Two valuations $v_1,v_2$ on $B$ are said to be \emph{equivalent} if $\ker{v_1}=\ker{v_2}=\p$ and the induced valuations $\bar{v}_1,\bar{v}_2$ on $k(\p)$ are equivalent in the classical sense \footnote{i.e. they have the same valuation ring or, equivalently, there is an order preserving group isomorphism between their images compatible with the valuations.}.\\
We will identify equivalent valuations. 

With this convention a valuation $v$ on $B$ with kernel $\p$ uniquely defines a valuation ring contained in $k(\p)$ by 
\begin{displaymath}
R_v = \{x\in k(\p) \mid \bar{v}(x)\le 1 \}.
\end{displaymath}

Hence a valuation $v$ on $B$ is equivalent to a diagram 
\begin{displaymath}
    \xymatrix {
       B  \ar[r]       & {k(\p)}           \\
         & R_v. \ar@{^{(}->}[u]
         }
\end{displaymath}

\begin{definition}
Let $B$ be a ring, $A$ a subring and $v$ a valuation on $B$.
We call $v$ an $A$-\emph{valuation} on $B$ if $v(a) \le 1$ for every $a \in A$. 
\end{definition}

Assume $v$ is an $A$-valuation with kernel $\p$. Set $ \q =\p \cap A $.
From the condition $v(a)\le 1  \; \forall a \in A$ we obtain a commutative diagram
\begin{displaymath}
\SelectTips{cm}{}
    \xymatrix {
        R_v \ar@{^{(}->}[r]         & {k(\p)}          & B\ar[l] \ar[r]^-{v}      & {\Gamma\cup\{\,0\}} \\
       {R_v\cap k(\q)} \ar@{^{(}->}[r] \ar@{^{(}->}[u] & {k(\q)} \ar@{^{(}->}[u]  & A.\ar[l]  \ar@{^{(}->}[u] \ar@{-<} `d[l] `[ll] [ll] \\}
\end{displaymath}
We conclude that every $A$-valuation $v$ on $B$ uniquely defines a commutative diagram
\begin{displaymath}
    \xymatrix {
       B  \ar[r]       & {k(\p)}           \\
       A  \ar[r]^-{\Phi}  \ar@{^{(}->}[u]  & R_v. \ar@{^{(}->}[u]
         }
\end{displaymath}
          
\noindent Conversely any such diagram defines an $A$-valuation $v$ on $B$ and
we are justified in identifying the $A$-valuation $v$ on $B$ with the 3-tuple $(\p,R_v,\Phi)$.

\subsection{The Auxiliary Space \textit{Spa(B,A)}} \label{Spa}
For completeness and consistency of notation we collect here results regarding valuation spectra.
The main reference of this subsection is \cite{hub1}.

\begin{definition}
For any pair of rings $A \subset B$ we set
\begin{displaymath}
Spa(B,A)=\{ A\text{-valuations on} \: B \}.
\end{displaymath}
\end{definition}
\noindent Fix a pair of rings $A \subset B$. \\
 
We provide $Spa(B,A)$ with a topology.
For any $a,b\in B$  set
\begin{displaymath}
U_{a,b}=\{v\in Spa(B,A) \mid v(a)\le v(b)\neq 0 \}.
\end{displaymath}
The topology is the one generated by the sub-basis $\{ U_{a,b} \}_{a,b\in B}$.

Given another pair of rings $A' \subset B'$ and a homomorphism of rings $\varphi :B \to B'$ that satisfies $\varphi(A) \subset A'$, composition with $\varphi$ gives rise to the pull back map 
$$\varphi^*:Spa(B',A') \to Spa(B,A).$$
Specifically given an $A'$-valuation $v=(\p,R_v,\Phi) \in Spa(B',A')$, then $v \circ \varphi$ is a valuation on $B$.
Since $\varphi(A) \subset A'$, $v \circ \varphi$ is an $A$-valuation.
So indeed $\varphi^*(v)= v \circ \varphi \in Spa(B,A)$.
Clearly its kernel is $\varphi^{-1}(\p)$.
Now, we have a commutative diagram
\begin{displaymath}
	\xymatrix{
	 & {\phantom{A}}B'{\phantom{A}} \ar[rr] & & k(\p) \\
	 {\phantom{A}}B{\phantom{A}} \ar[rr] \ar[ur]^{\varphi} & & k(\varphi^{-1}(\p)) \ar[ur] & \\
	 & A' \ar'[u][uu] \ar'[r][rr]^(0.4){\Phi} & & R_v \ar[uu]\\
	 A \ar[rr]^{\varphi^*(\Phi)} \ar[uu] \ar[ur] & & R_{v \circ \varphi} \ar[ur] \ar[uu] & .
	 }
\end{displaymath}
It is clear that $R_{v \circ \varphi}=R_v \cap k(\varphi^{-1}(\p))$ and that the ring map $\varphi^*(\Phi)$ is completely determined by $\Phi$ and $\varphi$.
To conclude, $\varphi^*$ takes the point $(\p,R_v,\Phi) \in Spa(B',A')$ to the point $(\varphi^{-1}(\p),R_v \cap k(\varphi^{-1}(\p)),\varphi^*(\Phi)) \in Spa(B,A)$.

Clearly $U_{\varphi(a),\varphi(b)} = {\varphi^*}^{-1}(U_{a,b})$ for any $a,b \in B$.
We obtain:

\begin{lemma} \label{sec-spa}
Let $A \subset B$ and $A' \subset B'$ be rings.
For a homomorphism $\varphi:B \to B'$ satisfying $\varphi(A) \subset A'$ the pull back map  $\varphi^*:Spa(B',A') \to Spa(B,A)$ is continuous.
\end{lemma}
        
Let $b,a_1,\ldots,a_ n\, \in B$ and assume that $b,a_1,\ldots,a_n$ generate the unit ideal.
Set $B'=B_b$.
Denote the canonical map $B \to B'$ by $\varphi_b$, and set $A'=\varphi_b(A)[\frac{a_1}{b},\ldots,\frac{a_n}{b}]$.
Obviously $A'\subset B'$ so  $Spa(B',A')$ is defined.
We also obtain a commutative diagram
          \begin{displaymath}
              \xymatrix {
                  B \ar[r]^-{\varphi_b}    & B'          \\
                  A \ar@{^{(}->}[u]\ \ar[r]      & A'  \ar@{^{(}->}[u]
                   }
           \end{displaymath}
which gives rise to the pull back map 
\begin{displaymath}
\varphi^*_b:Spa(B',A') \to Spa(B,A).
\end{displaymath}
In this case the pull back map is injective: if $v=(\p,R_v,\Phi)\in Spa(B,A)$ is in the image and $v'=(\p',R_{v'},\Phi')\in Spa(B',A')$ maps to $v$, then necessarily $ b \notin \p $ and $\p'=\p B'$. Hence $k(\p')= k(\p)$, from which follows that $R_{v'} = R_v$ and we have the diagram 
\begin{displaymath}
    \xymatrix {
        {Spec \:k(\p)} \ar[r] \ar [d]       & SpecB' \ar[r] & SpecA' \ar [d]       \\
        Spec \:R_v \ar[rr]^{\Phi} \ar [urr]^{\Phi'} &               & SpecA.
         }
\end{displaymath}
Since $SpecA' \to SpecA $ is separated, $\Phi'$ is unique by the valuative criterion for separateness, so  $v'=(\p',R_{v'},\Phi')$ is unique.

\noindent For such $A' \subset B'$ we regard $Spa(B',A')$ as a subset of $Spa(B,A)$. \\

Consider $v' \in Spa(B',A')$ as an element of $Spa(B,A)$.
As $\frac{a_i}{b} \in A'$ for all $i=1,\ldots ,n$, and $b \notin \ker{v'}$, we see that $v'(a_i)\le v'(b) \ne 0$ for every $i=1,\ldots ,n$.
Hence $Spa(B',A') \subset \{v \in Spa(B,A) \mid v(a_i)\le v(b) \quad \forall \; 1\le i \le n \}$.\\
Conversely, let $v' \in \{v \in Spa(B,A) \mid v(a_i)\le v(b) \quad \forall \; 1\le i \le n \}$.
Since $(b,a_1,\ldots ,a_n)=B$ there are $c_0,c_1, \ldots ,c_n$ in $B$ such that $1=c_0b+c_1a_1+\ldots +c_na_n$.
Applying $v'$ we obtain
\begin{multline*}
1=v'(1)=v'(c_0b+c_1a_1+\ldots +c_na_n)\le \\ \le  \max \{v'(c_0b), v'(c_1a_1),\ldots ,v'(c_na_n) \} \le v'(b) \max_{0 \le i \le n} \{ v'(c_i) \} 
\end{multline*} 
so necessarily $v'(b)\ne 0$.
Hence we can extend $v'$ to a valuation on $B'$ by $v'(\frac{b'}{b})=\frac{v'(b')}{v'(b)}$ for any $b' \in B$.
Since $v'(c)\le 1 \quad \forall \; c \in A$ we have $v'(A')\le 1$ so $v' \in Spa (B',A')$.
It follows that $\{v \in Spa(B,A) \mid v(a_i)\le v(b) \quad \forall \; 1\le i \le n \} \subset Spa(B',A')$, hence we have equality.
Furthermore we have
$$ Spa(B',A')  =\{v \in Spa(B,A) \mid v(a_i)\le v(b) \quad \forall \; 1\le i \le n \}=\cap^n_{i=1} U_{a_i,b}.$$
Hence $Spa(B',A')$ is an open subset of $Spa(B,A)$.
We obtain:

\begin{lemma} \label{rational domains in Spa}
Let $b,a_1,\ldots,a_ n\, \in B$ and assume that $b,a_1,\ldots,a_n$ generate the unit ideal.
Then
 $$Spa\left( B_b,\varphi_b(A)\left[ \frac{a_1}{b},\ldots,\frac{a_n}{b}\right] \right) = \{v \in Spa(B,A) \mid v(a_i)\le v(b) \quad \forall \; 1\le i \le n \}$$
and $Spa\left( B_b,\varphi_b(A)\left[ \frac{a_1}{b},\ldots,\frac{a_n}{b}\right] \right)$ is an open subset of $Spa(B,A)$.
\end{lemma}

\begin{definition}
We call such a set a \emph{rational domain} of $Spa(B,A)$ and denote it by $\R(\{a_1,\ldots ,a_n\} \slash b)$.
\end{definition}

\begin{remark} \label{inter. rational domains}
Let $a_0, \ldots ,a_n, a'_0, \ldots , a'_m \in B$ such that both $a_0, \ldots ,a_n$ and $a'_0, \ldots ,a'_m$ generate the unit ideal.
By Lemma \ref{rational domains in Spa} 
\begin{displaymath}
\R\left( \{ a_i \} ^{n}_{i=0} \slash a_0\right) = \{v \in Spa(B,A) \mid v(a_i)\le v(a_0) \quad \forall \; 1\le i \le n \} ,
\end{displaymath}
\begin{displaymath}
\R\left( \{a'_j \}^{m}_{j=0} \slash a'_0\right) = \{v \in Spa(B,A) \mid v(a'_i)\le v(a'_0) \quad \forall \; 1\le j \le m \} 
\end{displaymath}
and
  \begin{multline*}
\R\left( \{ a_i a'_j \} _{i,j} \slash a_0a'_0\right) = \\ \{v \in Spa(B,A) \mid v(a_ia'_j)\le v(a_0a'_0) \quad \forall \; 1\le i \le n \text{ and } 1\le j \le m \} .
  \end {multline*}
Hence
\begin{displaymath}
\R\left( \{ a_i \} ^{n}_{i=0} \slash a_0\right) \cap \R\left( \{a'_j \}^{m}_{j=0} \slash a'_0\right) =\R\left( \{ a_i a'_j \} _{i,j} \slash a_0a'_0\right) .
\end{displaymath}
\end{remark}

\begin{remark}
Let $b \in B$. If $b$ is nilpotent there is some $n > 0$ such that $b^n=0$.
For any valuation $v$ we have $0=v(b^n)=v(b)^n$ so $v(b)=0$.
From this its follows that for any $a_1, \ldots ,a_n \in B$ such that $(b,a_1, \ldots ,a_n)=B$ we have $\R(\{a_1,\ldots ,a_n\} \slash b)=\emptyset$.
If $b$ is not nilpotent, there is a prime ideal $\p$ not containing $b$.
Now for any $a_1, \ldots ,a_n \in B$ such that $(b,a_1, \ldots ,a_n)=B$ the rational domain $\R(\{a_1,\ldots ,a_n\} \slash b)$ contains the trivial valuation of $k(\p)$.
Concluding, we have 
\begin{displaymath}
\R(\{a_1,\ldots ,a_n\} \slash b)=\emptyset \Leftrightarrow b \text{ is a nilpotent element}.
\end{displaymath}
\end{remark}

By a \emph{rational covering} we mean the open cover defined by some $a_1, \ldots ,a_n \in B $ generating the unit ideal, that is the rational domains $ \Big\{ \R(\{a_i\}^{n}_{i=1} \slash a_j) \Big\}^{n}_{j=1}$.

In \cite{hub1}, Huber defines the valuation spectrum of a ring $B$
\begin{displaymath}
Spv(B) = \{ \text{valuations on} \: B \}.
\end{displaymath}
He provides it with the topology generated by the sub-basis consisting of sets of the form $\{ v | v(a) \le v(b) \neq 0 \}$ for all $a,b \in B$.
Huber proves in \cite[2.2]{hub1} that $SpvB$ is a spectral space.
Clearly our $Spa(B,A)$ is a subspace of Huber's $Spv(B)$.

\begin{lemma} \label{Spa is spectral}
The topological space $Spa(B,A)$ is spectral.
In particular it is quasi-compact and $T_0$. 	
\end{lemma}

\begin{proof}
Since a closed subspace of a spectral space is again spectral it is enough to show that $Spa(B,A)$ is closed in $Spv(B)$.
Following Huber's argument we just need to show that the set of binary relations $ | $ of $ \phi(SpvB) $ that satisfy $1 \mid a \; \forall a \in A$ is closed in  $ \phi(SpvB) $.
If $\mid' \in \phi(SpvB) - \phi(Spa(B,A))$ then there is an element $a \in A$ such that $1 \nmid' a$.
The set $V_{a}$ of binary relations satisfying $1 \nmid a$ contains $\mid'$ and is open in $\{ 0,1 \}^{B \times B}$.
Now, $V_{a} \cap \phi(SpvB)$ is open in $\phi(SpvB)$, contains $\mid'$ and
\begin{displaymath}
	( V_{a} \cap \phi(SpvB)) \bigcap \phi(Spa(B,A)) = \emptyset.
\end{displaymath}

\end{proof}

\subsection{The Space \textit{Val(B,A)}}

We say that a valuation $v:B \to \Gamma\cup\{0\}$ is \emph{bounded} if there is an element $\gamma \in \Gamma$ such that $v(b)<\gamma$ for every $b \in B$.

\begin{definition}
For any pair of rings $A \subset B$ we set
	\begin{equation*}
	Val(B,A)=\{v \in Spa(B,A) \mid \emph{$v$ is unbounded} \}
	\end{equation*}
with the induced subspace topology from $Spa(B,A)$.
\end{definition}

\noindent As a subspace of a $T_0$ space, $Val(B,A)$ is also a $T_0$ space.

For a valuation $v$ on $B$ with abelian group $\Gamma$ we denote by $c\Gamma_{v}$ the convex subgroup of $\Gamma$ generated by $\{ v(b) \mid b \in B \quad 1 \le v(b) \}$. 
For any convex subgroup $\Lambda$ we define a map $v':B \to \Lambda\cup\{0\}$ by $v'(b)=\begin{cases}
v(b) & \text{if}\; v(b)\in \Lambda \\
0 & \text{if}\; v(b) \notin \Lambda
\end{cases}$.\\
It is easily seen that $v'$ is a valuation on $B$ if and only if $c\Gamma_{v} \subset \Lambda$.

The valuation $v'$ obtained in this way is called a \emph{primary specialization} of $v$ associated with $\Lambda$ \cite[\S 1.2]{hub2}.
Note that $\ker v \subset \ker v'$.
It is easy to see that a valuation $v$ is not bounded if and only if it has no primary specialization other then itself.

\begin{lemma} \label{mapping of pri-spe}
Let $A \subset B$ and $A' \subset B'$ be rings, and $\varphi:B \to B'$ a  homomorphism satisfying $\varphi(A) \subset A'$.
Assume $v,w \in Spa(B',A')$ such that $w$ is a primary specialization of $v$.
Then $\varphi^*(w)$ is a primary specialization of $\varphi^*(v)$.
\end{lemma}

\begin{proof}
Assume that $v:B' \to \Gamma\cup\{\,0\}$ and that $\Lambda$ is the convex subgroup of $\Gamma$ associated with $w$.
Then for any $b' \in B'$ we have $w(b')=\begin{cases}
v(b') & \text{if}\; v(b')\in \Lambda \\
0 & \text{if}\; v(b') \notin \Lambda
\end{cases}$.
It now follows that for any $b \in B$ we have $w(\varphi(b))=\begin{cases}
v(\varphi(b)) & \text{if}\; v(\varphi(b))\in \Lambda \\
0 & \text{if}\; v(\varphi(b)) \notin \Lambda
\end{cases}$.
\end{proof}

For $v \in Spa(B,A)$, let $P_v$ be the the subset of all primary specializations of $v$.
Primary specialization induces a partial order on $P_v$ by the rule  $u \le w$ if $u$ is a primary specialization of $w$ for $u,w \in P_v$.

\begin{prop} \label{oredering}
For any $v\in Spa(B,A)$, the set $P_v$ of primary specializations of $v$ is totally ordered and has a minimal element.
\end{prop}

\begin{proof}
Let $v:B \to \Gamma\cup\{0\}$ be a valuation on $B$.
Let $w:B \to \Lambda\cup\{0\} $ and $ u:B \to \Delta\cup\{0\}$ be two distinct primary specializations of $v$.
We may regard $\Lambda$ and $\Delta$ as convex subgroups of $\Gamma$, so one is contained in the other.
As both $w$ and $u$ are primary specialization of $v$, both $\Lambda$ and $\Delta$ contain $c\Gamma_{v}$.
Assume $\Delta \subset \Lambda$. We want to show that $u$ is a primary specialization of $w$, i.e. $u(b)=\begin{cases}
w(b) & \text{if}\; w(b)\in \Delta \\
0 & \text{if}\; w(b) \notin \Delta
\end{cases}$.\\

For any $b \in B$ if $w(b) > 1$ then $v(b)=w(b)$, hence $c\Lambda_{w} \subset c\Gamma_{v}$.
Conversely if $v(b) > 1$ then since $c\Gamma_{v} \subset \Lambda$ we have $w(b)=v(b)$, hence $c\Lambda_{w} = c\Gamma_{v}$.
It follows that $\Delta$ is a convex subgroup of $\Lambda$ containing $c\Lambda_{w}$.

For any $b \in B$, if $w(b) \in \Delta$ then $w(b)=v(b) \in \Delta$.
It follows that $u(b)=v(b)=w(b)$.
If $w(b) \notin \Delta$ then either $w(b)=0$ or $0 \ne w(b) \in \Lambda$.
If $w(b)=0$ then $v(b) < \Lambda$, hence $v(b) < \Delta$ and $u(b)=0$.
If $w(b) \ne 0$ then $w(b)=v(b) \notin \Delta $, so $u(b)=0$.

The minimal element of $P_v$ is the primary specialization associated with $c\Gamma_{v}$.
\end{proof}

Next we give an algebraic criterion for a valuation $v \in Spa(B,A)$ to be in $Val(B,A)$.

\begin{lemma} \label{in val}
Let $v=(\p,R_{v},\Phi) \in Spa(B,A)$.\\
Then $v=(\p,R_{v},\Phi) \in Val(B,A)$ if and only if the canonical map $B\otimes_{A} R_{v} \to k(\p)$ is surjective.
\end{lemma}

\begin{proof}
Assume that $B\otimes _{A} R_{v} \to k(\p)$ is surjective. 
Since we assume that $\Gamma$ is generated by the image of $B-\,\p$, for any $1 < \gamma \in \Gamma$ there is $0 \neq f \in k(\p)$ satisfying $\gamma = \bar{v}(f)$.
Let $b_i \in B \:, r_i \in R_v$ for $i=1,\ldots,n$ such that $\sum b_i \otimes r_i \in B\otimes_{A} R_{v} $ maps to $f$ in $k(\p)$.
Assume $v(b_1)=max\{v(b_i)\} $.
Denote the image of $b_i$ in $k(\p)$ by $\bar{b_i}$.
Now, as $\bar{v}(r_i) \le 1$, we have
\begin{displaymath}
\gamma=\bar{v}(f)=\bar{v}(\sum \bar{b_i}r_i) \le max \{\bar{v}(\bar{b_i}r_i) \} \le max \{ \bar{v}(\bar{b_i}) \} = v(b_1).
\end{displaymath}
Hence $\gamma$ does not bound $v$.

Conversely, for $v=(\p,R_{v},\Phi) \in Val(B,A)$ we have a diagram
\begin{displaymath}
    \xymatrix {
                    &    &  k(\p) \ar [r]^-{\bar{v}}& {\Gamma\cup\{\,0\}} \\
        B \ar[r] \ar [urr] & B\otimes_{A}R_{v} \ar@{-->}[ur]  &   \\
        A \ar[r] \ar [u] &     R_{v}.  \ar [u] \ar [uur]        & 
         }
\end{displaymath}
For any $f \in k(\p)$, if $\bar{v}(f)\le 1$ then $f \in R_{v}$ and $1 \otimes f \in B\otimes_{A}R_{v}$ maps to $f \in k(\p)$.
Assume $\bar{v}(f) > 1$.
As $v \in Val(B,A)$ we see that $\Gamma = c\Gamma_{v}$, so there exists $d \in B$ satisfying $\bar{v}(f)\le v(d)$.
It follows that $f \slash \bar{d} \in R_{v}$ where $\bar{d}$ is the image of $d$ in $k(\p)$ and $d \otimes f \slash \bar{d} \in B\otimes_{A}R_{v}$ maps to $f \in k(\p)$.
\end {proof}

\begin{remark} \label{kk}
Since for any $A \subset B$ and $R_{v}$ we always have
\begin{displaymath}
    \xymatrix  {B\otimes_{\mathbb{Z}}R_{v} \ar@{>>}[r] & B\otimes_{A}R_{v}},
\end{displaymath} 
we can replace in the above lemma $B\otimes_{A}R_{v}$ with $B\otimes_{\mathbb{Z}}R_{v}$.
\end{remark}

\begin{remark} \label{kkk}
Equivalently we can say that $v$ is in $Val(B,A)$ if and only if 
\begin{displaymath}
 Spec\, k(\p) \to Spec B \times_{Spec A}Spec R_{v},
\end{displaymath}
 or equivalently 
\begin{displaymath}
Spec\, k(\p) \to Spec B \times Spec R_{v},
\end{displaymath} 
is a closed immersion.
\end{remark}

As we have seen, given another pair $A' \subset B'$ and a homomorphism
\begin{displaymath}
	\xymatrix{
	B \ar[r]^{\varphi} & B' \\
	A \ar[u] \ar[r] & A' \ar[u]
	}
\end{displaymath}
composition with $\varphi$ induces a map $\varphi^*:Spa(B',A') \to Spa(B,A)$.
However $\varphi^*$ does not necessarily restrict to a map $Val(B',A') \to Val(B,A)$.

\begin{lemma} \label{val to val}
Let
\begin{displaymath}
	\xymatrix{
	B \ar[r]^{\varphi} & B' \\
	A \ar[u] \ar[r] & A' \ar[u]
	}
\end{displaymath}
as above.
If the induced homomorphism $B \otimes_{A}A' \to B'$ is integral then composition with $\varphi$ induces a map $\varphi^*:Val(B',A') \to Val(B,A)$.
\end{lemma}

\begin{proof}
For any $v=(\p,R_{v},\Phi) \in Val(B',A')$ set 
\begin{displaymath}
\varphi^*(v)=v\circ \varphi=(\varphi^{-1}(\p),R_{v\circ \varphi},\varphi^*(\Phi) )=(\varphi^*(\p),\varphi^*(R_v),\varphi^*(\Phi)).
\end{displaymath}

\noindent We know that $\varphi^*(v) \in Spa(B,A)$.
In order to show that $\varphi^*(v) \in Val(B,A)$, by Lemma \ref{in val}, we need to show that $B\otimes _{A} \varphi^*(R_{v}) \to k(\varphi^*(\p))$ is surjective.

\noindent The homomorphism $\varphi$ gives rise to the digram
\begin{displaymath} 
\xymatrix {
               &   B' \ar [rr]     &      &  {k(\p)}      \\
      {\phantom{AA}}B{\phantom{A}} \ar[ur]^{\varphi} \ar [rr]      & & {k(\varphi^*(\p))}      \ar [ur]^{\bar{\varphi}} & \\
       &  A' \ar'[r]^{\Phi} [rr] \ar'[u] [uu]  &    & {R_{v}} \ar [uu] \\
      A \ar [ur] \ar [uu] \ar [rr]^{\varphi^*(\Phi)}    & & R_{\varphi^*(v)}\ar [ur] \ar [uu] &
         }
\end{displaymath} 
from which we see that there is a diagram
\begin{displaymath} 
       \xymatrix {
         B' \otimes_{A'}R_{v} \ar@{>>}[r]  &   {k(\p)}   \\
         B \otimes_{A}R_{\varphi^*(v)} \ar[u] \ar [r] &  {k(\varphi^*(\p))}. \ar[u]
         }
\end{displaymath}
The upper horizontal arrow is surjective by Lemma \ref{in val}.

For any $\alpha \in k(\varphi^*(\p))$, if $\overline{\varphi^*(v)}(\alpha) \le 1$ then $\alpha$ is already in $R_{\varphi^*(v)}$ (recall that $\overline{\varphi^*(v)}$ is the induced valuation on $k(\varphi^*(\p))$ ).

If  $\overline{\varphi^*(v)}(\alpha) > 1$, then there is $b' \in B'$ such that $\overline{\varphi^*(v)}(\alpha) \le v(b')$ since  $\bar{\varphi}(\alpha) \in k(\p)$ and $v$ is in $Val(B',A')$.
Since $B \otimes_{A}A' \to B'$ is integral we have $x_0, \ldots ,x_{n-1} \in Im\left( B \otimes_{A}A' \to B'\right) $ such that $b'^n + x_{n-1}b'^{n-1} + \ldots +x_0 = 0$.
As $v(0) = 0$ there is some $0 \le i \le n-1$ such that $v(b'^n) \le v(x_ib'^i)$.
It follows that $v(b') \le v(b')^{n-i} \le v(x_i)$.

Now, there are $a_1, \ldots ,a_m \in A'$ and $b_1, \ldots ,b_m \in B$ such that $\sum_j a_j\otimes b_j$ maps to $x_i$, so $v(b') \le max \{ v(a_j) \cdot v\circ\varphi(b_j)\} \le max \{v\circ\varphi(b_j)\}$.
The last inequality is due to the fact that $v(a) \le 1$ for every $a \in A'$.
Choose $b \in \{b_1,\ldots, b_m\}$ such that $\varphi^*(v)(b)=v\circ\varphi(b)=max \{v\circ\varphi(b_j)\}$.
Now we have $\overline{\varphi^*(v)}(\alpha) \le \varphi^*(v)(b)$.
Denoting the image of $b$ in $k(\varphi^*(\p))$ by $\bar{b}$, we have $\overline{\varphi^*(v)}(\alpha) \le \overline{\varphi^*(v)}(\bar{b})$ or in other words $\overline{\varphi^*(v)}(\frac{\alpha}{\bar{b}})\le 1$, hence $\frac{\alpha}{\bar{b}} \in R_{\varphi^*(v)}$ and $b \otimes \frac{\alpha}{\bar{b}} \in B \otimes_{A} R_{\varphi^*(v)}$ maps to $\alpha$.

\end{proof}

\subsection{Rational Domains}

Set $\X=Val(B,A)$. \\
For $b,a_1,\ldots,a_ n \, \in B$  generating the unit ideal we defined a rational domain in $Spa(B,A)$ as 
\begin{displaymath}
\mathcal{R}(\{a_1,\ldots ,a_n\} \slash b) = \{v \in Spa(B,A) \mid v(a_i)\le v(b) \}.
\end{displaymath}
We call the set  $\mathcal{R}(\{a_1,\ldots ,a_n\} \slash b) \cap \X$ a \emph{rational domain} in $\X$ and denote it by $\X(\{a_1,\ldots ,a_n\} \slash b)$.

\noindent Obviously  $\X\left( \{a_1,\ldots ,a_n\} \slash b\right) =Val\left( B_b,\varphi_b(A)\left[ \frac{a_1}{b},\ldots,\frac{a_n}{b}\right] \right)  $.

\begin{prop} \label{base}
The rational domains of $\X$ form a basis for the topology.
\end{prop}

\begin{proof}
Let $w \in \X$ and $U$ an open neighbourhood of $w$ in $Spa(B,A)$ (i.e. $U \cap \X$ is an open neighbourhood of $w$ in $\X$).
By the definition of the topology there is a natural number $N$ and $a_i,b_i \in B$ such that $v(a_i) \le v(b_i) \neq 0$ for each $i=1, \ldots, N$ such that $w \in \bigcap_i U_{a_i,b_i} \subset U$.
By taking the products $ \prod_{i} c_i$ where $c_i \in \{ a_i,b_i \}$ and $b= \prod_i b_i$, replacing $N$ with a suitable natural number, the $a_i$-s with the above products and shrinking $U$, we may assume that we have $a_1, \dots ,a_N , b \in B$ satisfying $w \in \cap_i U_{a_i,b} = U$.

As $w(b) \neq 0$ and $w \in \X$ we see that $w(b) \in c\Gamma_{w}$.
Since $w(b)^{-1}$ is not a bound of $w$, there exists $d \in B$ such that $w(b)^{-1} \le w(d)$.
It follows that $1=w(1) \le w(db)$ and 
\begin{multline*}
w \in  U \cap U_{1,db} = \{ v \in Spa(B,A) \mid v(a_i) \le v(b) \neq 0,1 \le v(db) \} =\\= \R(\{da_1,\ldots ,da_n,1\} \slash db).
\end{multline*}
Hence 
\begin{displaymath}
w \in \mathcal{R}(\{da_1,\ldots ,da_n,1\} \slash db) \cap \X=\X(\{da_1,\ldots ,da_n,1\} \slash db) \subset U \cap \X .
\end{displaymath}

It remains to show that the rational domains satisfy the intersection condition of a basis. However it follows from Remark \ref{inter. rational domains} that
\begin{displaymath}
\X\left( \{ a_i \} ^{n}_{i=1} \slash a_0\right) \cap \X\left( \{a'_j \}^{m}_{j=1} \slash a'_0\right) =\X\left( \{ a_i a'_j \} _{i,j} \slash a_0a'_0\right) .
\end{displaymath}
\end{proof}

In \cite{tem}, Temkin defines the \emph{semi-valuation ring} $S_{v}$ for a valuation $v= (\p,R_{v},\Phi)$ on $B$.
It is the pre-image of the valuation ring $R_{v}$ in the local ring $B_{\p}$. 
The valuation on $B$ induces a valuation on $S_{v}$.
We call $B_{\p}$ the \emph{semi-fraction ring} of $S_{v}$.

We briefly recall several properties of a semi-valuation ring (for details see \cite[\S2]{tem}).

\begin{remark} \label{k0}
If $S_{v}$ is a semi-valuation ring with semi-fraction ring $B_{\p}$ then
\begin{enumerate}[(i)]
	\item the maximal ideal $\p B_{\p}$ of $B_{\p}$ is contained in $S_{v}$.
	\item considering $v$ as a valuation on $B_{\p}$ or $S_{v}$ we have $\ker {v} = \p B_{\p}$.
	\item $(S_{v})_{ker v} = B_{\p}$.
	\item $S_{v} \slash ker v = R_{v}$, in particular $S_{v}$ is a local ring.
	\item for any pair $g,h \in S_{v}$ such that $v(g)\le v(h) \ne 0$ we have $g \in h S_{v}$.\label{k2}
	\item for any co-prime $g,h \in B_{\p}$ (i.e. $gB_{\p} + hB_{\p} = B_{\p}$), either $g \in h S_{v}$ or $h \in g S_{v}$.
	\item the converse of (vi) is also true: for a pair of rings $C \subset D$, if for any two co-prime elements $g,h \in D$ either $g \in h C$ or $h \in g C$ then there exists a valuation on $D$ such that $C$ is a semi-valuation ring of $v$ and $D$ is its semi-fraction ring.\label{k1}
\end{enumerate}
\end{remark}

\begin{cor}
Let $v=(\p,R_{v},\Phi) \in Spa(B,A)$.\\
Then $v=(\p,R_{v},\Phi) \in Val(B,A)$ if and only if the canonical map $B\otimes _{A} S_{v} \to k(\p)$ is surjective.
\end{cor}

\begin{proof}
As $S_v$ is the pull back of $R_v$ in $B_{\p}$, the canonical map $B\otimes _{A} S_{v} \to k(\p)$ factors through  $B\otimes _{A} R_{v}$.
Since $S_{v} \slash \p = R_{v}$ the ring map $B\otimes _{A} S_{v} \to B\otimes _{A} R_{v}$ is always surjective .
Hence $B\otimes _{A} S_{v} \to k(\p)$ is surjective if and only if  $B\otimes _{A} R_{v} \to k(\p)$ is surjective.
Now the result follows from Lemma \ref{in val}.
\end{proof}

Let us study how semi-valuation rings behave under pullback.

\begin{remark} \label{pullback semi-val}
Let $A \subset B$ and $A' \subset B'$ be rings and $\varphi:B \to B'$ a ring homomorphism such that 
\begin{displaymath}
	\xymatrix{
	B \ar[r]^{\varphi} & B' \\
	A \ar[u] \ar[r] & A' \ar[u]
	}
\end{displaymath}
commutes.
Consider the map $\varphi^*:Spa(B',A') \to Spa(B,A)$.
Let $v=(\p,R_{v},\Phi) \in Spa(B',A')$, then $\varphi^*(v)=v\circ\varphi$ and $R_{\varphi^*(v)}=R_v \cap k(\varphi^{-1}(\p))$.
We have a diagram
\begin{displaymath} 
\xymatrix {
               &   B'_{\p} \ar [rr]     &      &  {k(\p)}      \\
      B_{\varphi^{-1}(\p)} \ar[ur] \ar [rr]      & & {k(\varphi^*(\p))}      \ar [ur] & \\
       &  S_v \ar'[r] [rr] \ar'[u] [uu]  &    & {R_{v}} \ar [uu] \\
      S_{\varphi^*(v)}  \ar [uu] \ar [rr] \ar@{-->}
[ur]  & & R_{\varphi^*(v)}\ar [ur] \ar [uu] &
         }
\end{displaymath} 
from which we see that $B_{\varphi^{-1}(\p)} \to B'_{\p}$ restricts to a local homomorphism $S_{\varphi^*(v)} \to S_v$ of semi-valuation rings.
\end{remark}

Given two valuations $v=(\p,R_{v},\Phi), w=(\q,R_{w},\Psi) \in Spa(B,A)$ with $\p \subset \q$, we would like to know if $w$ a primary specialization of $v$.

\noindent Assume $ v:B \to \Gamma\cup\{\,0\}$ and $ w:B \to \Delta\cup\{\,0\}$.
Then it follows that $\Gamma = v(B_{\p}^{\times})$ and $\Delta = w(B_{\q}^{\times})$.
Since $\p \subset \q$ we have a canonical homomorphism $B_{\q} \to B_{\p}$.

\begin{lemma} \label{pri-spe}
With the above notation, $w$ is a primary specialization of $v$ if and only if the canonical homomorphism $B_{\q} \to B_{\p}$ restricts to a local homomorphism $S_w \to S_v$ of semi-valuation rings.
\end{lemma}

\begin{proof}
If $w$ is a primary specialisation of $v$, then $\Delta$ is a convex subgroup of $\Gamma$ containing $c\Gamma_v$ and
$w(b)=\begin{cases}
v(b) & \text{if}\; v(b)\in \Delta \\
0 & \text{if}\; v(b) \notin \Delta
\end{cases}$.

\noindent For any $x \in S_w \subset B_{\q}$ there are $b,s \in B$  , $s \notin \q=\ker w$ such that $x=\frac{b}{s}$.
Since $x \in S_w $ we have $w(b) \le w(s) \neq 0$.
Again from $\p \subset \q$ we see that $s \notin \p=\ker v$ and by the definition of a primary specialisation $w(s)=v(s)$.
If $w(b) \neq 0$ then, again $w(b)=v(b)$.
If $w(b) = 0$ then $v(b)\notin \Delta $.
Since $c\Gamma_v \subset \Delta $ then $v(b) < \Delta $, in particular  $v(b) < v(s)$.
In either case we have $v(b) \le v(s)$ so the image of $x=\frac{b}{s}$ in $B_{\p}$ is in $S_v$.
Furthermore if $x$ is in the maximal ideal of $S_w$ i.e. $w(x) <1$ then the same is true for its image in $S_v$, meaning that the homomorphism $B_{\q} \to B_{\p}$ restricts to a local homomorphism $S_w \to S_v$.

Conversely assume that we have a diagram
\begin{displaymath}
	\xymatrix{
	B_{\q} \ar[r] & B_{\p} \\
	S_w \ar[u] \ar[r] & S_v \ar[u] 
	}  
\end{displaymath}
with the bottom arrow a local homomorphism.
Define $\alpha: \Delta \to \Gamma$ by sending $w(x) \in \Delta$ ($x \in B_{\q}^{\times}$) to $v(x) \in \Gamma$.
Let $x_1,x_2 \in B_{\q}^{\times}$ such that $w(x_1)=w(x_2)$.
All elements of $ B_{\q}^{\times}$ are invertible so $x_1x^{-1}_2$ is also in $ B_{\q}^{\times}$.
Obviously  $w(x_1x^{-1}_2)=1$ so $x_1x^{-1}_2 \in S_w$.
Since  $S_w \to S_v$ is a local homomorphism we also have $v(x_1x^{-1}_2)=1$.
Hence $v(x_1)=v(x_2)$ and $\alpha$ is well defined.
Since $B_{\q}^{\times}$ is a multiplicative subset of $B_{\q}$ and valuations are multiplicative, $\alpha$ is also multiplicative.
As $1=w(1)=v(1)$ we get that $\alpha(1)=1$ i.e. $\alpha$ is a group homomorphism.\\
Note that if $x=\frac{b}{s} \in B_{\q} - S_w$ then its image under $B_{\q} \to B_{\p}$ lies in $B_{\p} - S_v$,
since if $w(x)>1$ and $v(x) \le 1$ then in particular $x^{-1} \in S_w$ and $w(x^{-1}) < 1$.
By the locality of the homomorphism we have $v(x^{-1}) < 1$.\\
But this is impossible since it would imply that $1=v(1)=v(xx^{-1}) <1$.
Now if $x \in B_{\q}^{\times}$ with $v(x)=1$ i.e. its image under $B_{\q} \to B_{\p}$ has value 1, then $x$ is already in $S_w$ and by locality of the homomorphism we have $w(x)=1$.
Hence $\alpha$ is an injection and we may regard $\Delta$ as a subgroup of $\Gamma$.\\
It is now clear that $w(b)=\begin{cases}
v(b) & \text{if}\; v(b)\in \Delta \\
0 & \text{if}\; v(b) \notin \Delta
\end{cases}$.
\end{proof}

\begin{obs}
Let $U=\X(\{a_1,\ldots ,a_n\} \slash b)$ be a rational domain.
Then for any valuation $v \in U$ we have $1 \le v(b)$.
For  if not then $\varphi(b) \in S_v \subset B_{\p}$ with $\varphi: B \to B_{\p}$.
It then follows that the ideal generated by $\varphi(b)$ is a proper ideal of the semi-valuation ring $S_v$.
As $v(a_i) \le v(b)$ we have $\varphi(a_i) \in S_v$.
Furthermore $\varphi(a_i) \in \varphi(b)S_v$ by Remark \ref{k0}.(\ref{k2}).
But $b,a_1,\ldots,a_ n \, \in B$  generate the unit ideal so $\varphi(b),\varphi(a_1),\ldots,\varphi(a_ n) \, \in B_{\p}$ generate the unit ideal which is a contradiction.

\end{obs}

\begin{theorem}[Transitivity of Rational Domains] \label{trasit}
Let $\X'$ be a rational domain in $\X$ and $\X''$ a rational domain in $\X'$.
Then $\X''$ is a rational domain in $\X$.
\end{theorem}

\begin{proof}
\begin{case}
$\X'=\X\left( \{ a_i \} ^{n}_{i=1} \slash 1\right)$ and $\X''=\X'\left( \{ b_j \} ^{m}_{j=1} \slash 1\right)$
\end{case}
Then we have $a_1,\ldots,a_ n  \in B$ and
\begin{displaymath}
\X'= \{ v \in \X \mid v(a_i) \le 1 \}=Val(B',A')
\end{displaymath}
with $B'=B$ and $A'=A[a_1,\ldots,a_n]$.
We also have $b_1,\ldots,b_ m  \in B'$ and
\begin{displaymath}
\X''=\X'\left( \{ b_j \} ^{m}_{j=1} \slash 1\right)=Val(B'',A'')
\end{displaymath}
with $B''=B'=B$ and $A''=A[a_1,\ldots,a_n][b_1,\ldots,b_m]$.
We see that $\X''= \{ v \in \X \mid v(a_i) \le 1 \text{ and } v(b_j) \le 1 \}=\X\left( \{ a_i \} \cup \{ b_j \} \slash 1\right)$.

\begin{case}
$\X'=\X\left( \{ a_i \} ^{n}_{i=1} \slash b\right)$ and $\X''=\X'\left( \{ 1 \}  \slash h\right)$
\end{case}
Now $\X'=\{ v \in \X \mid v(a_i) \le v(b) \}=Val(B',A')$ with $B'=B_b$ and $A'=\varphi_b(A)[\frac{a_1}{b},\ldots,\frac{a_n}{b}]$.
Also $\X''= \{ v \in \X' \mid 1 \le  v(h) \}$ for $h \in B_b$.
There exists $g \in B$ such that $\varphi_b(g) = b^kh$ for some $k \ge 0$.
Note that 
\begin{displaymath}
\X''=\X' \cap \{ v \in \X \mid v(b^k) \le v(g) \}.
\end{displaymath}
As we saw, if $v \in \X'$ then $1 \le v(b)$.
So $1 \le v(b^k)$ as well.
Then clearly $1 \le v(g)$ for any $v \in \X''$.
Thus we can write
\begin{displaymath}
\X''=\X\left( \{ a_i \} ^{n}_{i=1} \slash b\right) \cap \X\left( \{ b^k,1 \}  \slash g\right),
\end{displaymath}
so $\X''$ is an intersection of two rational domains which is, as we already saw, a rational domain.
\begin{case}
$\X'=\X\left( \{ a_i \} ^{n}_{i=1} \slash b\right)$ and $\X''=\X'\left( \{ h \}  \slash 1\right)$
\end{case}
Again $\X'=\{ v \in \X \mid v(a_i) \le v(b) \}=Val(B',A')$, but now $\X''= \{ v \in \X' \mid  v(h) \le 1 \}$ for $h \in B_b$.
Again there exists $g \in B$ such that $\varphi_b(g) = b^kh$ for some $k \ge 0$.
Now
\begin{displaymath}
\X''=\X' \cap \{ v \in \X \mid v(g) \le  v(b^k)  \},
\end{displaymath}
which can be rewritten as
\begin{displaymath}
\X''=\X\left( \{ a_i \} ^{n}_{i=1} \slash b\right) \cap \X \left( \{ g,1 \}  \slash b^k \right).
\end{displaymath}
\begin{case}
$\X'=\X\left( \{ a_i \} ^{n}_{i=1} \slash b\right)$ and $\X''=\X'\left( \{ h_j \} ^{m}_{j=1} \slash f\right)$
\end{case}
Finally we have $\X'=Val(B',A')$ and $\X''=\X'\left( \{ h_j \} ^{m}_{j=1} \slash f\right)$ for $f,h_1,\ldots,h_m \in B'$ generating the unit ideal.
As we saw
\begin{multline*}
\X''=\{ v \in \X' \mid v(h_j) \le v(f) \}=\\ = \{ v \in \X' \mid v(h_j) \le v(f)  \text{ and } 1 \le v(f)\} \subset \X'\left( \{ 1 \}  \slash f\right).
\end{multline*}
Furthermore $f$ is a unit in $B'_f$ so $f^{-1}h_1,\ldots,f^{-1}h_m $ are elements in $B'_f$.
Denoting $ \X'\left( \{ 1 \}  \slash f\right)$ by $\X^{(3)}$ we have
\begin{displaymath}
\X''=\X^{(3)}\left( \{ f^{-1}h_j \} ^{m}_{j=1} \slash 1\right) = \cap_j \X^{(3)}\left( \{ f^{-1}h_j \}  \slash 1\right)
\end{displaymath}
and repeated application of the previous cases gives the result.
\end{proof}

There is an obvious retraction $r: Spa(B,A) \to Val(B,A)$ given by sending every valuation $v$ to its minimal primary specialization. \label{retract}

\begin{prop} \label{ret-cont}
$r: Spa(B,A) \to Val(B,A)$ is continuous.
\end{prop}

\begin{proof}
Let $U$ be an open subset of $\X=Val(B,A)$. 
As the rational domains form a basis for the topology it is enough to consider the case when $U$ is a rational domain of $Val(B,A)$.
Let $a_1, \dots ,a_n , b \in B$  generating the unit ideal.
Set
\begin{displaymath}
U= \{ v \in Val(B,A) \mid v(a_i) \le v(b) \; i=1,\ldots,n \}=\X(\{a_1,\ldots ,a_n\} \slash b)
\end{displaymath} 
and
\begin{displaymath}
W= \{ v \in Spa(B,A) \mid v(a_i) \le v(b) \}=\R(\{a_1,\ldots ,a_n\} \slash b).
\end{displaymath} 
We claim that $r^{-1}(U)=W$. \\
Obviously $r^{-1}(U) \subset W$.
Let $w$ be a valuation in $W$ with value group $\Gamma$.
If $c\Gamma_{w}=\Gamma$ then $r(w)=w$, that is, $w \in Val(B,A)$ so $w \in W \bigcap Val(B,A) = U$.
If  $c\Gamma_{w} \subsetneqq \Gamma$ then $r(w)(a_i) \le r(w)(b)$ since $w(a_i) \le w(b)$.
It remains to show that $r(w)(b) \neq 0$.
If $r(w)(b) = 0$ then $w(b) < c\Gamma_{w}$ and so $w(a_i) < c\Gamma_{w}$ for every $i=1,\ldots,n$.
There are $c_0,c_1, \dots ,c_n  \in B$ such that $1=bc_0+a_1c_1+ \dots +a_nc_n$ and $w(1)=1 \in  c\Gamma_{w}$.
By convexity of $c\Gamma_{w}$ there is some $i$ such that $w(a_ic_i) \in c\Gamma_{w}$.
Thus $w(a_ic_i) \le w(bc_i) \in c\Gamma_{w}$, so $r(w)(bc_i) \neq 0$.
But since $r(w)(b) = 0$ we also get $r(w)(bc_i) = 0$, which is a contradiction.
We conclude that $r(w)(b) \neq 0$.
\end{proof}

\begin{cor} \label{Val is qcqs}
$Val(B,A)$ is quasi-compact and quasi-separated.
\end{cor}

\begin{proof}
As $Spa(B,A)$ is a quasi-compact space by Lemma \ref{Spa is spectral} and the retraction $r: Spa(B,A) \to Val(B,A)$ is continuous, $Val(B,A)$ is quasi-compact.
Any rational domain can be viewed as $Val(B',A')$ for suitable rings $A' \subset B'$, hence any rational domain is quasi-compact.
As we saw in Proposition \ref{base}, the intersection of two rational domains is again a rational domain so in particular it is quasi-compact.
Since the rational domains form a basis of the topology, any quasi-compact open subset of $Val(B,A)$ can be viewed as a finite union of rational domains.
Now the intersection of any two quasi-compact open subsets of $Val(B,A)$
is also a finite union of rational domains, thus quasi-compact so $Val(B,A)$ is quasi-separated.
\end{proof}

\begin{example}[An Affine Scheme] \label{e1}

Consider $Val(B,B)$.

\noindent Let $v=(\p,R_{v},\Phi) \in Val(B,B)=\X$, then $v$ is an unbounded valuation on $B$ such that $v(B) \le 1$.
The only way this could be is if $v$ is a trivial valuation (i.e. $\Gamma = \{ 1 \}$).
Hence there is a $1-1$ correspondents between points of $Val(B,B)$ and prime ideals of $B$, that is points of $SpecB$.
As for the topology:
  \begin{multline*}
   \X(\{a_1,\ldots ,a_n\} \slash b)=Val\left( B_b,\varphi_b(B)\left[ \frac{a_1}{b},\ldots,\frac{a_n}{b}\right] \right) =\\=Val(B_b,B_b) \leftrightarrow SpecB_b = D(b)
  \end{multline*}
So there is a homeomorphism $Val(B,B) \simeq SpecB$.

\end{example}

For later use we define two canonical maps of topological spaces
\begin{displaymath}
\sigma: Spec B \to \X \text{ and } \tau: \X \to Spec A.
\end{displaymath}
For $\p \in Spec B$ we set $\sigma(\p)$ to be the trivial valuation on $k(\p)$, which is indeed in $\X=Val(B,A)$ since its valuation ring is $k(\p)$ and $B\otimes _{A} k(\p) \to k(\p)$ is indeed surjective.
For $v=(\p,R_{v},\Phi) \in \X$ we denote the maximal ideal of $R_v$ by $\m_v$ and set $\tau(v)=\Phi^{-1}(\m_v) \in SpecA$.

\begin{prop} \label{sigtau}
\begin{enumerate}
\item The composition $\tau \circ \sigma: SpecB \to SpecA$ is the morphism corresponding to the inclusion of rings $A \subset B$.
\item \label{sigtau sigma} $\sigma$ is continuous and injective.
\item $\tau$ is continuous and surjective.
\end{enumerate}
\end{prop}

\begin{proof}
$(1)$ Given a prime $\p$ in $B$, $\sigma(\p)=(\p,k(\p),\Phi)$.
The maximal ideal of the valuation ring $k(\p)$ is the zero ideal, so
\begin{displaymath}
\tau(\sigma(\p))=\ker \Phi = \ker \left(  A \to B \to k(\p) \right) = \p \cap A.
\end{displaymath}

$(2)$ Let $U=\X(\{a_1,\ldots ,a_n\} \slash b)$ be a rational domain in $\X$ and $D(b)$ a basic open set in $Spec B$.
For any $\p \in D(b)$ we have $\sigma(\p)(b)=1$ and $\sigma(\p)(a_i)=0 \text{ or }1$, so $\sigma(D(b)) \subset \X(\{a_1,\ldots ,a_n\} \slash b)$.

Conversely if $v \in \X(\{a_1,\ldots ,a_n\} \slash b)$ and there is $\p \in Spec B$ that maps to $v$ then we must have $\ker(v)=\p$ and $v$ is trivial on $k(\p)$.
Hence $\sigma$ is injective and $\sigma^{-1}(\X(\{a_1,\ldots ,a_n\} \slash b)) = D(b)$.

$(3)$ For a basic open set $D(a) \subset Spec A$ and $v=(\p,R_{v},\Phi) \in \X$, then  $\tau(v) \in D(a)$ is the same as $\Phi(a) \notin \m_v$ or equivalently $v(a) \ge 1$.
In this case $v(a)=1$ (since $v(A) \le 1$) and 
\begin{displaymath}
v \in \X(\{1\} \slash a) = \{ w \in Val(B,A) \mid w(a)=1 \}
\end{displaymath}
so  $\tau^{-1}(D(a))=\X(\{1\} \slash a)$.

As for surjectivety, first consider a maximal ideal $\q \in SpecA$.
Since the morphism $SpecB \to SpecA$ is schematically dominant there is $\p'' \in SpecB$ such that $A \cap \p'' \subset \q$. 
Take $\p$ to be a maximal prime of $B$ with this property.
The image $i(A)$ of $A$ under $i:B \to k(\p)$ is a subring of $k(\p)$.
The extended ideal $i(\q)i(A)$ is a proper ideal, since if $1 \in i(\q)i(A)$ then there are $a_1, \ldots, a_n \in \q$ and  $b_1, \ldots, b_n \in A$ such that $1-\sum_i a_ib_i \in ker (i)=\p$.
As $1-\sum_i a_ib_i \in A$ we have $1-\sum_i a_ib_i \in A \cap \p \subset \q $, but since $\sum_i a_ib_i \in \q$ we get that $1 \in \q$ which is a contradiction.
By \cite[VI \S4.4]{ZSII}  there is a valuation ring $R_v$ of $k(\p)$ containing $i(A)$ such that its maximal ideal $\m_v$ contains $i(\q)i(A)$.
This gives us a valuation $v=(\p,R_v,\Phi=i|_{A}) \in Spa(B,A)$.
BY the retraction we obtain a valuation $v'=r(v) \in Val(B,A)$ with kernel $\p' \supset \p$.
If $\p \ne \p'$ then $\q$ is strictly contained in $A \cap \p'$, by the choice of $\p$.
Since $\q$ is a maximal ideal we have $1 \in A \cap \p'$ which is a contradiction.
Hence $\p' = \p$ and $v' = v$ that is $v=(\p,R_v,\Phi) \in Val(B,A)$.
Now $\tau(v) = \Phi^{-1}(\m_v) \supset \q$ but since $\q$ is maximal we have
$\tau(v) =  \q$.

Now, let $\q \in SpecA$ be some prime.
Then $\q A_{ \q }$ is the maximal ideal of $A_{ \q }$, and since $A_{ \q }$ is flat over $A$, we have $A_{ \q } \subset B_{ \q } = B \otimes_A A_{ \q }$.
By the previous case there is a valuation $v \in Val(B_{ \q },A_{ \q })$ that $\tau:Val(B_{ \q },A_{ \q} ) \to SpecA_{ \q }$ maps to $ \q A_{ \q } $.
The canonical homomorphisms
\begin{displaymath}
	\xymatrix{
	B \ar[r] & B_{\q} \\
	A \ar[u] \ar[r] & A_{\q} \ar[u]
	}
\end{displaymath}
induces by Lemma \ref{val to val} a morphism $Val(B_{ \q },A_{ \q} ) \to \X$.
As $ \q A_{ \q } \in Spec A_{\q} $ is pulled back to $\q \in SpecA $, the image of $v$ in $Val(B,A)$ is mapped by $\tau$ to  $\q$.
\end{proof}

\begin{remark} \label{tau2}
From the proof we see that for any $a \in A$ we have 
\begin{displaymath}
\tau^{-1}(D(a))=Val(B_a,A_a)=\X(\{1\} \slash a),
\end{displaymath}
and for every rational domain $\X(\{a_1,\ldots ,a_n\} \slash b) \subset \X$ we have
\begin{displaymath}
\sigma^{-1}(\X(\{a_1,\ldots ,a_n\} \slash b)) = D(b)
\end{displaymath}
\end{remark}

\begin{remark} \label{tau}
A point $v=(\p,R_{v},\Phi) \in \X$ is (as said before) a diagram
\begin{displaymath}
	\xymatrix{
	B \ar[r] & k(\p) \\
	A \ar[u] \ar[r]^{\Phi} & R_v. \ar[u]
	}
\end{displaymath}
Also there is a unique semi-valuation ring $S_v$ associated to the point (namely the pull-back of $R_v$ to $B_{\p}$), and the diagram factors through the pair $S_v \subset B_{\p}$ i.e.
\begin{displaymath}
	\xymatrix{
	B \ar[r] & B_{\p} \ar[r] & k(\p) \\
	A \ar[u] \ar[r] & S_v \ar[u] \ar[r] & R_v. \ar[u]
	}
\end{displaymath}
Since the maximal ideal of the valuation ring $R_v$ is pulled back to the maximal ideal of the semi-valuation ring $S_v$, we may rephrase the definition of $\tau$ as the pull-back of the maximal ideal of the semi-valuation ring $S_v$ to $A$.
Equivalently we can say that $\tau(v)$ is the image of the unique closed point of $SpecS_v$ under the map $SpecS_v \to SpecA$.
\end{remark}

\subsection{Rational Covering}

Any open cover of $\X$ can be refined to a cover of $\X$ consisting of rational domains, since the rational domains form a basis for the topology.
Furthermore there is a finite sub-cover of $\X$ consisting of rational domains, as $\X$ is quasi-compact.
Next we show that we can always refine this cover to a rational covering, that is there are elements $T=\{ a_1, \ldots ,a_N \} \subset B$ generating the unit ideal such that the rational domains $\{ \X( T \slash a_j) \}_{ 1 \le j \le N}$ form a refinement of the finite sub-cover.

\begin{prop}
Any finite open cover of $Val(B,A)$ consisting of rational domains can be refined to a rational covering.
\end{prop}

\begin{proof}
Let $\{ U_i\}_{i =1}^{N}$ be a finite open cover of $\X = Val(B,A)$  consisting of rational domains, i.e. for every $i=1,\ldots,N$ we have $a_1^ {(i)}, \ldots  ,a_{n_i}^ {(i)}\in B $ generating the unit ideal and 
\begin{displaymath}
U_i = \X( \{a_j^{ {(i)} }\}_{1\le j \le n_i} \slash a_1^{(i)} )=\{ v \mid v(a_j^ {(i)}) \le v(a_1^{(i)}) \quad j=1,\ldots,n_i \}.
\end{displaymath}

For every $1 \le k \le n_i $ denote $V_{i,k} = \X( \{a_j^{ {(i)} }\}_{1\le j \le n_i} \slash a_k ^{(i)} )$.
Note that $V_{i,1}=U_i$ and $\X = \cup_{k=1} ^{n_i} V_{i,k}$ for each $i=1,\ldots,N$.

\noindent Set 
\begin{displaymath}
I = \{ (r_1,\ldots,r_N) \in \mathbb{N}^N \: | \: 1\le r_i \le n_i \; i=1,\ldots ,N \}
\end{displaymath}
and
\begin{displaymath}
I' = \{ (r_1,\ldots,r_N) \in I  \: | \: r_i=1 \: \text{for some }\: i  \}.
\end{displaymath} 

For $(r_1,\ldots,r_N) \in I$ we denote 
\begin{displaymath}
V_{(r_1,\ldots,r_N)} = \cap_{1\le i \le N} V_{i,r_i} \text{ and } a_{(r_1,\ldots,r_N)} = a_{r_1}^ {(1)} \cdot a_{r_2}^ {(2)} \cdot \ldots \cdot a_{r_N}^ {(N)}.
\end{displaymath}

\noindent Note that $V_{(r_1,\ldots,r_N)} = \{ v \in \X \: | \: v(a_{\alpha}) \le v(a_{(r_1,\ldots,r_N)}) \; \forall \alpha \in I \}$. \\
We claim that for any $(r_1,\ldots,r_N) \in I'$ we have
\begin{displaymath}
V_{(r_1,\ldots,r_N)} = \{ v \in \X \: | \: v(a_{\alpha}) \le v(a_{(r_1,\ldots,r_N)}) \; \forall \alpha \in I' \}.
\end{displaymath} 
Given $(r_1,\ldots,r_N) \in I'$, by definition we have
\begin{displaymath}
V_{(r_1,\ldots,r_N)} \subset \{ v \in \X \: | \: v(a_{\alpha}) \le v(a_{(r_1,\ldots,r_N)}) \; \forall \alpha \in I' \}.
\end{displaymath} 
Conversely, if $w \in \{ v \in \X \: | \: v(a_{\alpha}) \le v(a_{(r_1,\ldots,r_N)}) \; \forall \alpha \in I' \}$, then $w \in U_{i_0}$ for some $i_0$ since $\{ U_i\}_{i =1}^{N}$ is a cover of $\X$.
For simplicity we assume that $i_0=1$, so $w(a_k ^{(1)})\le w(a_1 ^{(1)})$ for every $1 \le k \le n_1 $. \\
Now for any $(j_1,\ldots,j_N) \in I - I'$ we have 
\begin{multline*}
w(a_{(j_1,\ldots,j_N)}) = w(a_{j_1}^ {(1)} \cdot a_{j_2}^ {(2)} \cdot \ldots \cdot a_{j_N}^ {(N)}) \le \\ \le w(a_{1}^ {(1)} \cdot a_{j_2}^ {(2)} \cdot \ldots \cdot a_{j_N}^ {(N)})= w(a_{(1,j_2,\ldots,j_N)}).
\end{multline*}
As $(1,j_2,\ldots ,j_N) \in I'$ by assumption we have $w(a_{(1,j_2,\ldots,j_N)}) \le w(a_{(r_1,\ldots,r_N)})$.
It follows that
\begin{displaymath}
w(a_{(j_1,\ldots,j_N)}) \le w(a_{(1,j_2,\ldots,j_N)}) \le w(a_{(r_1,\ldots,r_N)}),
\end{displaymath} 
hence $w \in V_{(r_1,\ldots,r_N)}$.

Finally note that \begin{itemize}
 \item $ \{ a_{\alpha} \}_{\alpha \in I'} $ generate the unit ideal of  $B$.
 \item $V_{\alpha}= \X(\{ a_{\beta}  \} _{\beta \in I'} \slash a_{\alpha} )$ for every $\alpha \in I'$.
 \item $\X = \cup_{\alpha \in I'}V_{\alpha} $.
 \item for $(r_1,\ldots,r_N) \in I'$ if $r_i=1$ then $V_{(r_1,\ldots,r_N)} \subset U_i$.
 \end{itemize}
\end{proof}

\subsection{Sheaves on \textit{Val(B,A)}}

\subsubsection{$\M_{\X}$ and $\OO_{\X}$} 
\label{sigma}

We now define two sheaves on $\X = Val(B,A)$, both making $\X$ a locally ringed space.

\begin{notation}
\begin{itemize}
	\item For a pair of rings $C \subset D$ we denote the integral closure of $C$ in $D$ by $Nor_{D}C$.
	\item For quasi-compact quasi-separated schemes $Y,X$ and an affine morphism $f:Y \to X$, we denote the integral closure of $\OO_X$ in $f_{*}\OO_Y$ by $Nor_{f_{*}\OO_Y}\left( \OO_X\right) $ or $Nor_{Y}\left( \OO_X\right) $.
	\item In the situation above we denote $Nor_{Y}X=\mathbf{Spec}_X(Nor_{Y}\OO_X)$ and the canonical morphism $Nor_YX \to X$ by $ \nu $.
\end{itemize}
\end{notation}

\begin{lemma} \label{nor-bir}
Let $A \subset B$ be rings.
Denote $X=SpecA$ and $Y=SpecB$.
Then $\X=Val(B,A)=Val(B,Nor_BA)$ and the canonical map $\tau:\X \to X$ factors through $Nor_YX$.
\end{lemma}

\begin{proof}
The canonical morphism 
\begin{displaymath}
\nu: Nor_YX \to X
\end{displaymath} 
is an integral, hence universally closed and separated.
Thus for any valuation $v=(\p,R_{v},\Phi) \in \X$ we obtain a diagram  by  the valuative criterion (abusing notation and denoting by $\Phi$ both $A \to R_v$ and the induced morphism $SpecR_v \to SpecA$) 
  \begin{displaymath} \xymatrix {
        {Spec \:k(\p)} \ar[r] \ar [d]       & Y \ar[r] & Nor_YX \ar [d]^{\nu}       \\
        Spec \:R_v \ar[rr]^{\Phi} \ar@{-->} [urr]^{\exists !\Phi'} &        & X.
         }
\end{displaymath}
That is, we obtain a unique $(\p,R_{v},\Phi') \in Spa(B,Nor_BA)$.
As $v \in \X$ the morphism $Spec\, k(\p) \to Y \times Spec R_{v}$
is a closed immersion by Remark \ref{kkk}.
It follows, again form Remark \ref{kkk}, that $(\p,R_{v},\Phi') \in Val(B,Nor_BA)$.
As the lifting of $\Phi$ is unique we have $\X=Val(B,Nor_BA)$.\\
It is clear that the diagram of topological spaces 
\begin{displaymath}
\xymatrix{
Val(B,Nor_BA) \ar@{=}[d] \ar[r]^-{\tau} &  Nor_YX \ar[d]^{\nu} \\
Val(B,A) \ar[r]^-{\tau} & X
}
\end{displaymath}
commutes.
\end{proof}

As the rational domains form a basis for the topology of $\X$ it is enough to define the sheaves only over the rational domains. Let $U=\X(\{a_1,\ldots ,a_n\} \slash b)=Val(B',A')$ where, as before,  $B'=B_b$ and $A'=\varphi_b(A)\left[ \frac{a_1}{b},\ldots,\frac{a_n}{b}\right] $. We define two presheaves $\M_{\X}$ and $\OO_{\X}$ on the rational domains of $\X$ by the rules 
\begin{equation*}
U \mapsto \M_{\X}(U)=B' \qquad \qquad U \mapsto \OO_{\X}(U)=Nor_{B'}A'.
\end{equation*}
Clearly  $\OO_{\X}(U)\subset \M_{\X}(U)$.

\begin{theorem} \label{sheaves}
With the above notation, the presheaves $\M_{\X}$ and $\OO_{\X}$ are sheaves on the rational domains of $\X$.
\end{theorem}

\begin{proof}

Denote $Y=SpecB$.
We know that the sets $\{ D(b)\}_{b\in B} $ form a base for the topology of $Y$.
Recall that we defined a map $\sigma:Y \to \X$ such that for every rational domain $\X\left( \{a_1,\ldots ,a_n\} \slash b\right)  \subset \X $ we have $\sigma^{-1}\left( \X\left( \{a_1,\ldots ,a_n\} \slash b\right) \right)  = D(b)$ (Remark \ref{tau2}).
Now, by the definition of $\M_{\X}$, for every rational domain $U \subset \X$ we have an isomorphism of rings
\begin{displaymath}
B_b=\M_{\X}(U) \simeq \sigma_{*}\OO_{Spec B}(U)=\OO_{Spec B}(D(b))=B_b.
\end{displaymath} 
Let $V \subset U \subset \X$ be two rational domains.
Suppose $U=\X\left( \{a_1,\ldots ,a_n\} \slash b\right) $ and $V=\X\left( \{f_1,\ldots ,f_m\} \slash g\right) $.
Then $D(g)=\sigma^{-1}(V)$ and $D(b)=\sigma^{-1}(U)$.
For any $\p \in D(g)$ we have $\sigma(\p) \in V \subset U$, hence $\p \in D(b)$.
In other words we have a diagram
\begin{displaymath}
  \xymatrix {
        D(g) \ar[r] \ar [d]  & D(b) \ar[r] \ar[d] & Y \ar[d]^{\sigma}       \\
        V \ar[r]   &   U \ar[r]     & \X.
         }
\end{displaymath}
From the diagram we see that the restrictions of $\M_{\X}$ commute with the restrictions of $\OO_{SpecB}$.
We conclude that we have an isomorphism of presheaves between $\M_{\X}$ and $\sigma_{*}\OO_{Spec B}$ as presheaves on the rational domains.
Since $\OO_{Spec B}$ is a sheaf on $Y$, its restriction to a base of the topology of $Y$ is also a sheaf.
It follows that $\M_{\X}$ is a sheaf on the rational domains of $\X$.

Let $U$ be a rational domain in $\X$ and $\{V_i\}$ an open covering of $U$ consisting of rational domains.
Let $s_i \in \OO_{\X}(V_i)$ be sections satisfying $s_i|_{V_i\cap V_j} = s_j|_{V_i \cap V_j}$ for every pair $i,j$.
We already know that $\M_{\X}$ is a sheaf so there is a unique element $s \in \M_{\X}(U)$ such that $s|_{V_i} = s_i$ for each $i$.
We want to show that $s$ is in $\OO_{\X}(U)$.
We may assume that $U=\X$ and that $\{V_i\}$ is a rational covering corresponding to $b_1, \ldots , b_r$, that is $b_1, \ldots , b_r \in B$, none of which are nilpotent, generating the unit ideal of $B$ and $V_i=\X\left( \{b_j\}_{j}\slash b_i\right) $.

We denote $Y=SpecB$ , $X=SpecA$ , $B_i=B_{b_i}$ , $A_i=\varphi_i(A)\left[ \{\frac{b_j}{b_i}\}_{j}\right] $ (where $\varphi_i$ is the canonical homomorphism $B\to B_i$), $Y_i=SpecB_i$ and $X'_i=SpecA_i$.
Then $V_i=Val(B_i,A_i)$.

Now, $E=\sum_iAb_i$ is a finite $A$-module contained in $B$.
Using the multiplication in $B$, we define $E^d$ as the image of $E^{\otimes d}$ under the map $B^{\otimes d} \to B$.
Then $E^d$ is also a finite $A$-module contained in $B$ for any $d \ge 1$.
Denoting $E^0 =A$ we obtain a graded $A$-algebra $E'=\oplus_{d \ge 0}E^d$ and a morphism $X'=Proj\left( E' \right) \to X$. 
The affine charts of $X'$ are given by $SpecA'_i$, where $A'_i$ is the zero grading part of the localization $E'_{b_i}$.
Clearly $A'_i \subset B_i$.
Denoting $I_d=\ker(E^d \to B \to B_i)$, we have $A'_i = \lim_{d \to \infty} {b^{-d}_i \cdot E^d \slash I_d}$ which is exactly $A_i$.
This means we have open immersions
\begin{displaymath}
  \xymatrix {
        Y_i \ar@{^{(}->}[r] \ar [d]  & Y  \ar[dr]  &     \\
        X'_i \ar@{^{(}->}[r]   &   X' \ar[r]     & X.
         }
\end{displaymath}
As $Y= \cup Y_i$ the schematically dominant morphisms $Y_i \to X'_i$ glue to a schematically dominant morphism $Y \to X'$ over $X$.\\ 
Furthermore, for each $i$ we have a commutative diagram
\begin{displaymath}
  \xymatrix {
        Y_i \ar[r]^{\sigma} \ar [d]  & V_i  \ar[r]^{\tau} \ar[d]  & X'_i \ar[d]    \\
        Y \ar[r]^{\sigma}   &   \X  \ar[r]^{\tau}     & X.
         }
\end{displaymath}
Denoting the normalization $X''=Nor_Y X'$ and taking the canonical morphism  $\nu:X'' \to X'$, then by Lemma \ref{nor-bir}  we have a diagram over $X$
\begin{displaymath}
  \xymatrix {
        Y \ar[r]  \ar[rrd] & \X \ar[r]^{\tau} \ar[rd]^{\tau}  & X''\ar[d]^{\nu}    \\
        &     &  X'  .   
                 }
\end{displaymath}
We denote $X''_i=\nu^{-1}(X'_i)$. 
Now, by construction $s_i \in \OO_{X''}(X''_i) = \tau^{*}\OO_{\X}(V_i)$ and  $s_i|_{X''_i\cap X''_j} = s_j|_{X''_i \cap X''_j}$ for every pair $i,j$.
Since $\OO_{X''}$ is a sheaf, they glue to a section $s \in \OO_{X''}(X'')=\tau^{*}\OO_{\X}(\X)$.
\end{proof}

Note that the above construction yields the same topological space and the same sheaves for $A \subset B$ and for $Nor_{B}A \subset B$.

\subsubsection{The stalks}

\begin{prop} \label{stalks}
For any point $v=(\p,R_{v},\Phi) \in \X $, the stalk $\M_{\X,v}$ of the sheaf $\M_{\X}$ is isomorphic to  $B_{\p}$ and the stalk $\OO_{\X,v}$ of the sheaf $\OO_{\X}$ is isomorphic to the semi-valuation ring $S_{v}$.
\end{prop}
\begin{proof}
Fix a point $v= (\p,R_{v},\Phi) \in \X $.
The inclusion of sheaves $\OO_{\X}\subset \M_{\X}$ gives an inclusion of stalks $\OO_{\X,v}\subset \M_{\X,v}$.\\
By the definition of a semi-valuation ring we have a diagram
\begin{displaymath} 
\xymatrix {
        B \ar[r]          &   B_{\p} \ar [r]            &  {k(\p)}      \\
        A \ar [r] \ar [u] &   S_{v} \ar [r] \ar [u]     & {R_{v}} . \ar [u]
         }
\end{displaymath}
For any rational domain  $v \in W = \X(\{a_i\}_{i}\slash b)=Val(B_1,A_1)$ (we may assume that $A_1$ integrally closed in $B_1$) we have a unique factorization
\begin{displaymath} 
\xymatrix {
        B \ar[r]          &   \M_{\X}(W)=B_1 \ar [rr]  \ar@{-->} [dr]    &      &  {k(\p)}      \\
         & & B_{\p} \ar [ur] & \\
        A \ar [r] \ar [uu] &  \OO_{\X}(W)=A_1 \ar'[r] [rr] \ar [uu] \ar@{-->} [dr] &    & {R_{v}}. \ar [uu] \\
         & & S_{v}\ar [ur] \ar [uu] &
         }
\end{displaymath} 
Taking direct limits we get a unique diagram for the stalks
\begin{displaymath} 
       \xymatrix {
         \M_{\X,v} \ar[r]          &   B_{\p} \ar [r]            &  {k(\p)}      \\
         \OO_{\X,v} \ar [r] \ar [u] &   S_{v} \ar [r] \ar [u]     & {R_{v}} \ar [u]
         }
\end{displaymath}
and $v$ induces a valuation in $Val(\M_{\X,v},\OO_{\X,v})$.

Let $\gamma,\eta$ be co-prime elements in $\M_{\X,v}$, i.e. there are elements $\rho,\tau \in \M_{\X,v}$ such that $\gamma\rho+\eta\tau=1$.
It follows from Proposition \ref{base} that the intersection of a finite number of rational domains is again a rational domain.
Hence there is a rational domain $U=Val(B',A')$ with $g,h,r,t \in \M_{\X}(U)=B'$ such that $g,h,r,t$ are representatives of $\gamma,\eta,\rho,\tau$ respectively.
Then $gr+ht$ is a representative of $1 \in \M_{\X,v}$.
So there is a rational domain $V=Val(B'',A'') \subset U $ such that 
\begin{displaymath}
gr+ht|_{V}=1 \in \M_{\X}(V)=B''.
\end{displaymath}
Then we get that $g|_{V},h|_{V} \in \M_{\X}(V)=B''$ are representatives of $\gamma,\eta$ and are co-prime.
Furthermore $v$ induces (canonically) a valuation on $B''$ which has the same valuation ring as $v$.
By the transitivity of rational domains we may assume that $V=Val(B,A)$ and that $g,h \in B$ are co-prime and are representatives of $\gamma,\eta \in \M_{\X,v}$ respectively.
If $v(g)\le v(h)$ then $Val\left( B_h,\varphi_h(A)\left[ \frac{g}{h}\right] \right) $ is a rational domain and 
\begin{displaymath}
v \in Val\left( B_h,\varphi_h(A)\left[ \frac{g}{h}\right] \right)  \subset Val(B,A).
\end{displaymath}
Hence $v(\gamma \slash \eta)\le 1$, so $\gamma \in \eta \OO_{\X,v}$.
Conversely if $v(g)\ge v(h)$ then by the same reasoning $\eta \in \gamma \OO_{\X,v}$.
By Remark \ref{k0}.(\ref{k1})  $\OO_{\X,v}$ is a semi-valuation ring and $\M_{\X,v}$ its semi-fraction ring.
It now follows, also from Remark \ref{k0}, that for $\mathfrak{m}=\ker{v}$ in $\OO_{\X,v}$ we have $\M_{\X,v}=( \OO_{\X,v})_{\mathfrak{m}}$ and $\OO_{\X,v} \slash \mathfrak{m}=R_{v}$, the valuation ring of $v$ in $k(\p)$.
Hence $\OO_{\X,v}=S_{v}$ and $\M_{\X,v}=B_{\p}$.

\end{proof}

\begin{remark} \label{sigma2}
For any point $\p \in Spec B$ the stalks of the point $ \sigma(\p) \in \X $ are $$\M_{\X,\sigma(\p)}=\OO_{\X,\sigma(\p)}=B_{\p}.$$  
\end{remark}

\section{Birational Spaces}

\subsection{Pairs of Rings and Pairs of Schemes}

\begin{definition}

\begin{enumerate}[(i)] \label{a}
	\item A \emph{pair of rings} $(B,A)$ is a ring $B$ and a sub-ring $A$.
	\item A homomorphism of pairs of rings $\varphi :(B,A) \to (B',A')$ is a ring homomorphism $\varphi: B \to B'$ such that $\varphi(A) \subset A'$.
	\item \label{a3} A homomorphism of pairs of rings $\varphi :(B,A) \to (B',A')$ is called \emph{adic} if the induced homomorphism $B \otimes_{A}A' \to B'$ is integral.
	\item The \emph{relative normalization} of a pair of rings $(B,A)$ is the induced pair of rings $(B,Nor_B A)$ together with a canonical homomorphism of pairs $$\nu=id_B:(B,A) \to (B,Nor_B A).$$
	\item A \emph{pair of schemes} $(Y \stackrel{f}{\to} X)$ or $(Y,X)$ is a pair of quasi-compact, quasi-separated schemes $Y,X$ together with an affine and schematically dominant morphism $f:Y \to X$.
	\item A morphism of pairs of schemes $g:(Y',X') \to (Y,X)$ is a pair of morphisms $g=(g_Y,g_X)$ forming a commutative diagram
	\begin{displaymath} 
       \xymatrix {
         Y' \ar[r]^{g_Y} \ar[d]^{f'}  &  Y \ar[d]^{f}   \\
         X' \ar [r]^{g_X}  &   X.
         }
\end{displaymath}
	\item \label{a4} A morphism of pairs of schemes $g:(Y',X') \to (Y,X)$ is called \emph{adic} if the induced morphism of schemes
	\begin{displaymath}
    \xymatrix@1  {Y' \to Y\times_{X}X'}
	\end{displaymath}
is integral.
	\item The \emph{relative normalization} of a pair of schemes $(Y,X)$ is the induced pair of schemes $(Y,Nor_Y X)$ together with a canonical morphism of pairs $$\nu=(id_Y,\nu_X):(Y,Nor_Y X) \to (Y,X),$$ where $Nor_YX=\mathbf{Spec}_X(Nor_{\OO_Y}\OO_X)$ and $\nu_X$ is the canonical morphism $Nor_Y X \to X$.
\end{enumerate}
\end{definition}

We denote the category of pairs of rings with their morphisms by \textit{pa-Ring} and the category of pairs of schemes with their morphisms by \textit{pa-Sch}.

\begin{definition}
Let $(Y \stackrel{f}{\to} X)$ be a pair of schemes.
Given an open (affine) subscheme $X' \subset X$ its preimage $Y'=f^{-1}(X')$ is an open (affine) subscheme of $Y'$.
The restriction of $f$ to $Y'$ makes $(Y',X')$ a pair of schemes.
We call $(Y',X')$  an \emph{open (affine) sub-pair of schemes}.
An \emph{affine covering} of the pair $(Y,X)$ is a collection of open sub-pairs $\{ (Y_i,X_i) \}$ such that $\{ X_i \}$ are affine and cover $X$ (then necessarily their preimages $\{ Y_i \}$ are affine and cover $Y$).
\end{definition}

\begin{lemma}
Assume the elements $b,a_1,\ldots,a_ n\, \in B$ generate the unit ideal.
Set $B'=B_b$ .
Let $\varphi_b:B\to B_b$ be the canonical map and denote  $A'=\varphi_b(A)[\frac{a_1}{b},\ldots,\frac{a_n}{b}]$.
Then the homomorphism of pairs of rings $\varphi_{b}: (B,A) \to (B',A')$ is adic.
\end{lemma}

\begin{proof}
Since $b,a_1,\ldots,a_n$ generate the unit ideal of $B$
\begin{displaymath}
B \otimes_{A} A'= \varphi_b(B) \left[   \frac{a_1}{b},\ldots,\frac{a_n}{b} \right]   =B'.
\end{displaymath}

\end{proof}

\begin{lemma} \label{adic morph}
\begin{enumerate}[(i)]
\item Composition of adic morphisms of pairs of schemes is adic.
\item  Let $g :(Y',X') \to (Y,X)$ be an adic morphism of pairs of schemes and $(V,U)$ an open sub-pair of $(Y,X)$. Then the restriction of $g|_{(g_Y^{-1}(V),g_X^{-1}(U))}:(g_Y^{-1}(V),g_X^{-1}(U)) \to (V,U)$ is adic.
\item Let $g :(Y',X') \to (Y,X)$ be a morphism of pairs of schemes and $\{(V_i,U_i)\}$ an affine covering of $(Y,X)$.
If all the restrictions $(g_Y^{-1}(V_i),g_X^{-1}(U_i)) \to (V_i,U_i)$ are adic, then $g$ is adic.
\end{enumerate}
\end{lemma}

\begin{proof}
$(i)$ Let $g :(Y',X') \to (Y,X)$ and $h :(Y'',X'') \to (Y',X')$ be adic morphisms of pairs of schemes.
Then $g \circ h : (Y'',X'') \to (Y,X)$ is a morphisms of pairs of schemes associated to the diagram
\begin{displaymath}
	\xymatrix{
	Y'' \ar[r]^{h} \ar[d] & Y' \ar[r]^{g} \ar[d] & Y \ar[d] \\
	X'' \ar[r]  & X' \ar[r]  & X. 
	}
\end{displaymath}
We want to show that the induced morphism $Y''\to Y \times_X X''$ is integral.

This morphism factors as $Y''\to Y' \times_{X'} X'' \to Y \times_X X''$.
As $h$ is adic the first arrow is integral.
Now, $g$ is also adic, so $Y' \to Y \times_X X'$ is also integral.
Taking the base change of this morphism by the morphism $Y \times_X X'' \to Y \times_X X' $ we get that $Y' \times_{X'} X'' \to Y \times_{X} X'$ is also integral.
Hence the composition $Y''\to Y \times_X X''$ is indeed integral.

$(ii)$ We denote $U' = g_{X}^{-1}(U) \subset X'$ and $V' = f'^{-1}(U')=g_{Y}^{-1}(V) \subset Y'$.
Now  $V \times_U U' =V \times_X X' $ is an open subset of $ Y \times_X X'$ and we have a pull back diagram
\begin{displaymath}
	\xymatrix{
	V' \ar[r] \ar[d] & V \times_X X' \ar[d] \\
	Y' \ar[r] & Y \times_X X'.
	}
\end{displaymath}
Since the bottom arrow is integral so is the top arrow.

$(iii)$ For each $i$ denote $U'_i$ and $V'_i$ as in $(ii)$ and $g_i$ the restriction of $g$ to $(V'_i,U'_i) \to (V_i,U_i)$.
Assume that $g_i:(V'_i,U'_i) \to (V_i,U_i)$ is adic for each $i$.
Then for each $i$ the map $V'_i \to V_i \times_{U_i} U'_i=V_i \times_X X' $ is integral.
Now $Y'=\cup V'_i$ and $Y \times_X X'=\cup (V_i \times_X X') $.
Since the property of being integral is local on the base we have that $Y'\to Y \times_X X'$ is integral.
\end{proof}

\subsection{The \textit{bir} Functor}

\begin{definition}
\begin{enumerate}[(i)]
	\item A \emph{pair-ringed space} $(\X,\M_{\X},\OO_{\X})$ is a topological space $\X$ together with a sheaf of pairs of rings $(\M_{\X},\OO_{\X})$ such that both $(\X,\M_{\X})$ and $(\X,\OO_{\X})$ are ringed spaces.
	\item A morphism of pair-ringed spaces 
\begin{displaymath}
(h,h^\sharp):(\X,\M_{\X},\OO_{\X}) \to (\Y,\M_{\Y},\OO_{\Y})
\end{displaymath}
is a continuous map $h:\X \to \Y$ together with a morphism of sheaves of pairs of rings 
\begin{displaymath}
h^{\sharp}:(\M_{\Y},\OO_{\Y}) \to (h_{*}\M_{\X},h_{*}\OO_{\X})
\end{displaymath}
such that both
\begin{displaymath}
(h,h^\sharp):(\X,\M_{\X}) \to (\Y,\M_{\Y}) \text{ and } (h,h^\sharp):(\X,\OO_{\X}) \to (\Y,\OO_{\Y})
\end{displaymath}
are morphisms of ringed spaces.

\end{enumerate}

\end{definition}

In Section 2 we constructed a pair-ringed space $(\X,\M_{\X},\OO_{\X})$ from a pair of rings $(B,A)$, namely $Val(B,A)$.
From now on for any pair of rings $(B,A)$ by $Val(B,A)$ we mean the pair-ringed space $(Val(B,A),\M_{Val(B,A)},\OO_{Val(B,A)})$.

\begin{definition}
\begin{enumerate}[(i)]
	\item An \emph{affinoid birational space} is a pair-ringed space $(\X,\M_{\X},\OO_{\X})$ isomorphic to $Val(B,A)$ for some pair of rings $(B,A)$.
	\item A pair-ringed space $(\X,\M_{\X},\OO_{\X})$ is a \emph{birational space} if every point $x \in \X$ has an open neighbourhood $U$ such that the induced subspace $(U,\M_{\X}|_U,\OO_{\X}|_U)$ is an affinoid birational space.
	\item A morphism of birational spaces
	\begin{displaymath}
(h,h^\sharp):(\X,\M_{\X},\OO_{\X}) \to (\Y,\M_{\Y},\OO_{\Y})
\end{displaymath}
 is a morphism of pair-ringed spaces such that $(h,h^\sharp):(\X,\OO_{\X}) \to (\Y,\OO_{\Y})$ is a map of locally ringed spaces (but not necessarily $(h,h^\sharp):(\X,\M_{\X}) \to (\Y,\M_{\Y})$; see Example \ref{M not local} below).
\end{enumerate}
\end{definition}

\noindent We denote the category of affinoid birational spaces with their morphisms by \textit{af-Birat} and the category of birational spaces with their morphisms by \textit{Birat}.

\begin{example} \label{e2}
Given a ring $B$ we have a homeomorphism  of topological spaces $Val(B,B) \simeq SpecB$ (Example \ref{e1}). 
From Remark \ref{sigma2} it is clear that considered as an affinoid birational space $Val(B,B)$ is exactly $(SpecB,\OO_{SpecB},\OO_{SpecB})$.
\end{example}

A scheme is locally isomorphic to an affine scheme so we obtain

\begin{cor}
\begin{enumerate}
\item Any scheme $(X,\OO_X)$ can be viewed as a birational space $$Val(X,X)=(X,\OO_X,\OO_X).$$
\item Any pair of schemes $(Y,X)$ induces a birational space $Val(Y,X)$.
\end{enumerate}
\end{cor}

\begin{remark}
\begin{enumerate}
\item For a pair of schemes $(Y,X)$, the points of $Val(Y,X)$ are 3-tuples $(y,R,\Phi)$ such that $y$ is a point in $Y$, $R$ is a valuation ring of the residue field $k(y)$ and $\Phi$ is a morphism of schemes $SpecR \to X$ making the diagram
\begin{displaymath}
	\xymatrix{
	Y \ar[r] \ar[d]^f & Speck(y) \ar[d] \\
	X \ar[r]^-{\Phi} & SpecR
	} 
\end{displaymath}
commute, and $ Speck(y) \to Y \times_X SpecR$ is a closed immersion \cite[\S3]{tem}.
\item Any affine covering $\{ (Y_i,X_i) \}$ of $(Y,X)$ gives rise to a covering of the birational space $Val(Y,X)$ consisting of affinoid birational spaces $\{Val(Y_i,X_i) \}$.
\end{enumerate}
\end{remark}

\begin{example}
For a finitely generated field extension $K / k$ there is an obvious natural map, homeomorphic onto its image, from the Zariski-Riemann space $RZ(K / k)$ as defined by Zariski to our $Val(K,k)$ \cite[Chapter VI \S17]{ZSII}.
Furthermore $Val(K,k)$ consists of the image of $RZ(K / k)$ together with the trivial valuation on $K$ which is a generic point.
At a valuation $v$ the stalk is the pair of rings $(K,R_v)$.
\end{example}

\begin{theorem}
There is a contra-variant functor  \emph{bir} from the category \emph{pa-Rings} of pairs of rings to the category \emph{af-Birat} of affinoid birational spaces.
\end{theorem}

\begin{proof}

We already saw the construction of an affinoid birational space $Val(B,A)$ from a pair of rings $(B,A)$, we set $(B,A)_{bir}=Val(B,A)$.

For two pairs of rings $(B_1,A_1),$  $(B_2,A_2)$ and a homomorphism of pairs of rings $\varphi :(B_1,A_1) \to (B_2,A_2)$  we define the map of topological spaces $\varphi_{bir}$ by the composition
\begin{displaymath} 
       \xymatrix {
        Spa(B_2,A_2) \ar[r]^{\varphi^*}  &   Spa(B_1,A_1) \ar [d]^{r}    \\
        Val(B_2,A_2) \ar@{-->}[r]^{\varphi_{bir}} \ar@{^{(}->}[u] &    Val(B_1,A_1)  
         }
\end{displaymath}
where $\varphi^*$ is the pull back map defined in section \ref{Spa} and $r$ is the retraction defined in section \ref{retract}.

We saw that both $\varphi^*$ (Lemma \ref{sec-spa}) and the retraction (Lemma \ref{ret-cont}) and  are continuous so $\varphi_{bir}$ is continuous.

For $v=(\p,R_{v},\Phi)  \in Val(B_2,A_2)$ we have 
\begin{align*}
 \varphi^*(v)  =v \circ \varphi & \in Spa(B_1,A_1)  \\
\varphi_{bir}(v) = r(v \circ \varphi )= & w=(\q,R_{w},\Psi)  \in Val(B_1,A_1).
\end{align*}

Since $w$ is a primary specialisation of the pullback valuation $\varphi^*(v)=v\circ \varphi$ there is a natural homomorphism of the stalks
\begin{displaymath}
	\xymatrix{
\M_{Val(B_1,A_1),w}=(B_1)_{\q} \ar[r] & (B_1)_{\varphi^{-1}(\p)} \ar[r] & (B_2)_{\p}= \M_{Val(B_2,A_2),v} \\
\OO_{Val(B_1,A_1),w}=S_w \ar[u] \ar[r] & S_{v\circ \varphi} \ar[u] \ar[r] & S_v=\OO_{Val(B_2,A_2),v} .\ar[u]
}
\end{displaymath}
As we saw in Remark \ref{pullback semi-val} and Lemma \ref{pri-spe}  both bottom arrows are local homomorphism, hence so is their composition. 
Hence we obtain a morphism of affinoid birational spaces 
\begin{equation*}
\varphi_{bir}: (B_2,A_2)_{bir} \to (B_1,A_1)_{bir}.
\end{equation*}

It is obvious that \textit{bir} respects identity homomorphisms.
As for composition, let
\begin{displaymath}
(B_1,A_1) \stackrel{\varphi}{\to} (B_2,A_2) \stackrel{\psi}{\to} (B_3,A_3)
\end{displaymath} 
be homomorphisms of pairs of rings.
For $v=(\p,R_{v},\Phi)  \in Val(B_3,A_3)$ we have $(\psi \circ \varphi)^*(v)=\varphi^*(\psi^*(v))$ as elements of $Spa(B_1,A_1)$.
By construction $(\psi \circ \varphi)_{bir}(v)$ is a primary specialization of $(\psi \circ \varphi)^*(v)$.
Also, $\psi_{bir}(v)$ is a primary specialization of $\psi^*(v)$ as elements of $Spa(B_2,A_2)$, and $\varphi_{bir}(\psi_{bir}(v))$ is a primary specialization of $\varphi^*(\psi_{bir}(v))$ as elements of $Spa(B_1,A_1)$.
It follows from Lemma \ref{mapping of pri-spe} that  $\varphi^*(\psi_{bir}(v))$ is a primary specialization of $\varphi^*(\psi^*(v))$.
Hence both $\varphi_{bir}(\psi_{bir}(v))$ and $(\psi \circ \varphi)_{bir}(v)$ are primary specializations of $(\psi \circ \varphi)^*(v)$.
They are also both minimal primary specializations, since thy are elements of $Val(B_1,A_1)$.
By Proposition \ref{oredering} we have $\varphi_{bir}(\psi_{bir}(v)) = (\psi \circ \varphi)_{bir}(v)$.

Concluding, we obtained a functor 
\begin{displaymath}
bir : \textit{pa-Rings}^{op} \to \textit{af-Birat.}
\end{displaymath}

\end{proof}

The following example shows that the homomorphism on the stalks of $\M$ can indeed be not local.

\begin{example} \label{M not local}
Let $K$ be an algebraically closed field.
Consider $A=A'=B=K[T]$ and $B'=K[T,T^{-1}]$.
Let $\varphi:(B,A) \to (B',A')$ be the obvious map.
Clearly $\varphi$ is not adic.
Passing to birational spaces, we have
\begin{displaymath}
\varphi_{bir}:\X'=Val(K[T,T^{-1}],K[T]) \to Val(K[T],K[T])=\X.
\end{displaymath}
As we saw in Example \ref{e2} $\X=(\mathbb{A}^1,\OO_{\mathbb{A}^1},\OO_{\mathbb{A}^1})$.
Let $\eta$ be the generic point of $\mathbb{A}^1$ and of $Spec K[T,T^{-1}] \subset \mathbb{A}^1$.
Let $v$ be the valuation in  $\X'$ corresponding to the valuation ring $R_v=K[T]_{(T)} \subset K(T)=k(\eta)$.
It is indeed in $\X'$ as $K[T,T^{-1}]\otimes K[T]_{(T)} \to K(T)$ is surjective.
Since $K[T]\otimes K[T]_{(T)} \to k(\eta)=K(T)$ is not surjective the pullback $v \circ \varphi$ is not in $\X$.
Its primary specialization $w$ is the trivial valuation on $k(\p)=K$ for the ideal $\p=(T)$.
The stalks are $\M_{\X',v}=K(T)$ and $\M_{\X,w}=K[T]_{(T)}$.
The induced homomorphism of stalks is the obvious injection
\begin{displaymath}
K[T]_{(T)} \to K(T)
\end{displaymath}
which is not a local homomorphism.
\end{example}

\begin{lemma} \label{sigtau2}
Let $(B,A),$  $(B',A')$ be two pairs of rings and let  $\varphi :(B,A) \to (B',A')$ be a homomorphism of pairs.
Then the diagram of topological spaces
\begin{displaymath}
	\xymatrix{
	SpecB' \ar[r]^{\varphi^*} \ar[d]^{\sigma'} & SpecB \ar[d]^{\sigma} \\
	Val(B',A') \ar[d]^{\tau'} \ar[r]^{\varphi_{bir}} & Val(B,A) \ar[d]^{\tau} \\
	SpecA' \ar[r]^{\varphi^*|_{SpecA'}}  & SpecA
	}
\end{displaymath}
commutes.

\end{lemma}

\begin{proof}
For $\p \in Spec B'$, $\sigma \circ \varphi^* (\p) \in Val(B,A)$ is the trivial valuation on the residue field $k(\varphi^{-1}(\p))$.
Also $\sigma'(\p)$ is the trivial valuation on the residue field $k(\p)$.
The homomorphism $\varphi:B \to B'$ induces an injection $k(\varphi^{-1}(\p)) \to k(\p)$ , so the composition $  \varphi_{bir}\circ \sigma'(\p)$ is just the trivial valuation on the residue field $k(\varphi^{-1}(\p))$.

Now, given $v=(\p,R_{v},\Phi) \in Val(B',A')$ let $\varphi_{bir}(v)=w=(w,R_{w},\Psi)$.
Denote by $\m_v$ the maximal ideal of $S_v$ and by $\m_w$ the maximal ideal of $S_w$.
The homomorphisms $\Phi:A' \to R_v$ and $\Psi:A \to R_w$ factor through $S_v$ and $S_w$ respectively, denote these by $\Phi':A' \to S_v$ and $\Psi':A \to S_w$.
By Remark \ref{tau} we have $\tau'(v)=\Phi'^{-1}(\m_v)$ and $\tau(\varphi_{bir}(v))=\Psi'^{-1}(\m_w)$.
Since the induced homomorphism $S_v \to S_w$ is local we have $\varphi^{-1}(\Phi'^{-1}(\m_v)) \cap A =\Psi'^{-1}(\m_w)$.
Hence 
\begin{displaymath}
\tau(\varphi_{bir}(v))=\Psi'^{-1}(\m_w)=\varphi^{-1}(\Phi'^{-1}(\m_v)) \cap A = \varphi^*|_{SpecA'}(\tau'(v)).
\end{displaymath}
\end{proof}

Given two pairs of rings $(B,A) , (B',A')$ and a morphism of  affinoid birational spaces $h:Val(B,A) \to Val(B',A')$, by taking global sections we get a homomorphism of rings $\varphi$ which makes a commutative diagram
\begin{displaymath}
       \xymatrix {
        B' \ar[r]^{\varphi}         &   B  \\
        Nor_{B'}A' \ar [r] \ar [u] &   Nor_{B} A. \ar [u] 
         }
\end{displaymath}
Composition with the inclusion $A' \subset Nor_{B'}A'$ gives a morphism of pairs $\varphi: (B',A') \to (B,Nor_{B} A)$.
Since $Val(B,Nor_{B}A)=Val(B,A)$ by applying the $bir$ functor we obtain another morphism of  affinoid birational spaces $\varphi_{bir}:Val(B,A) \to Val(B',A')$.

\begin{theorem}  \label{important theorem}
The \emph{bir} functor is an anti-equivalence of the category of pairs of rings \emph{pa-Rings} localized at the class of relative normalizations and the category of affinoid birational spaces \emph{af-Birat}.
\end{theorem}

\begin{proof}
Denote the class of relative normalization homomorphisms by $M$.
By Lemma \ref{nor-bir}  $bir$ factors through $\emph{pa-Rings}_M$ and by definition of \textit{af-Birat} it is essentially surjective.
It remains to show that 
\begin{displaymath}
bir : \textit{pa-Rings}_M \to \textit{af-Birat}
\end{displaymath}
is full and faithful. 

We start by proving fullness.
Let $(B',A') , (B,A)$ be two pairs of rings and $h$ a morphism of  affinoid birational spaces $h:(B,A)_{bir} \to (B',A')_{bir}$.
We may assume that $A'$ and $A$ are integrally closed in $B'$ and $B$ respectively.

Taking global sections we get a diagram 
\begin{displaymath} 
       \xymatrix {
        B' \ar[r]^{\varphi}         &   B    \\
        A' \ar [r] \ar [u] &    A \ar [u] 
         }
\end{displaymath}
i.e. $\varphi$ is a homomorphism of pairs of rings $\varphi :(B',A') \to (B,A)$.
We claim that $\varphi_{bir}=h$.

Given a valuation $v=(\p,R_{v},\Phi) \in Val(B,A)$ we denote  $h(v)=w=(\q,R_{w},\Psi) \in Val(B',A')$.
Passing to stalks we have a diagram of pairs of rings

\begin{displaymath} 
       \xymatrix {
        (B',A') \ar[r]^{\varphi}  \ar [d]   &   (B,A) \ar [d]     \\
(B'_{\q}, S_{w}) \ar [r]^{h^{\sharp}_{v}} & (B_{\p}, S_{v}) . 
           }
\end{displaymath}
Denote by $\n = \p B_{\p}$ the maximal ideal of $B_{\p}$ , $\n' = \q B'_{\q}$ the maximal ideal of $B'_{\q}$, $\m_v$ the maximal ideal of $S_v$ and $\m_w$ the maximal ideal of $S_w$.
Since $B'_{\q}$ is a local ring ${h^{\sharp}_{v}}^{-1}(\n) \subset \n'$.
Pulling back to $B'$ we see that $\varphi^{-1}(\p) \subset \q$.
Hence the bottom line of the above diagram can be factored as
\begin{displaymath}
\xymatrix{
	B'_{\q} \ar[r] & B'_{\varphi^{-1}(\p)}=(B'_{\q})_{{h^{\sharp}}^{-1}_{v}(\n)} \ar[r] & B_{\p} \\
	 S_{w} \ar[u] \ar[r] & S_{v \circ \varphi} \ar[r] \ar[u] & S_{v} .\ar[u]
	 }
\end{displaymath}
By Remark \ref{pullback semi-val} and the definition of morphisms of birational spaces we see that the bottom left arrow is a local homomorphism.
By Lemma \ref{pri-spe} $w$ is a primary specialisation of the pullback valuation $v \circ \varphi$.
As $w$ is already in $Val(B',A')$ it has no primary specialisation other than itself.
By Proposition \ref{oredering} the primary specialisations of $v \circ \varphi$ are linearly ordered and we conclude that $r(v \circ \varphi)=w$ where $r$ is the retraction, or in other words $\varphi_{bir}(v)=h(v)$.

As for faithfulness, given two homomorphisms of pairs of rings $\varphi,\psi:(B',A') \to (B,A)$ such that $\varphi_{bir}=\psi_{bir}$ we obtain a diagram
\begin{displaymath} 
  \xymatrix { 
 SpecB \ar[d]^{\sigma} \ar@<0.5ex>[r]^{\varphi^*} \ar@<-0.5ex>[r]_{\psi^*} & SpecB' \ar[d]^{\sigma'}    \\
  Val(B,A)  \ar[r]^{\varphi_{bir}}_{\psi_{bir}}  & Val(B',A') .   
         }
\end{displaymath}

It follows from Proposition \ref{sigtau} (\ref{sigtau sigma}), Remark \ref{sigma2}, the definition of a morphism of birational spaces and Lemma \ref{sigtau2} that $\varphi^*=\psi^*$ as morphisms of schemes $SpecB \to SpecB'$, and hence $\varphi=\psi$ as homomorphisms of rings $B' \to B$.

If furthermore $A'$ and $A$ are integrally closed in $B'$ and $B$ respectively, then we also have that $\varphi=\psi$ as homomorphisms of pairs of rings $(B',A') \to (B,A)$.
\end{proof}

As an immediate corollary we obtain
 
\begin{cor} \label{important cor}
Let $(B',A') , (B,A)$ be two pairs of rings and $h$ an isomorphism of  affinoid birational spaces $h:(B,A)_{bir} \stackrel{\sim}{\to} (B',A')_{bir}$.
Then $B \simeq B'$ and $Nor_BA=Nor_BA'$ as subrings of $B$.
\end{cor}

%\begin{proof}
%From Theorem \ref{important theorem} we have an isomorphism of pairs of rings $\varphi :(B',Nor_{B'}A') \to (B, Nor_BA)$ such that $\varphi_{bir}=h$.
%\end{proof}
Next we want to characterize adic morphisms. 
\begin{prop}
A homomorphism of pairs of rings $\varphi :(B_1,A_1) \to (B_2,A_2)$ is adic if and only if the induced morphism of ringed spaces 
$$(\varphi_{bir},\varphi_{bir}^\sharp):(Val(B_2,A_2),\M_{Val(B_2,A_2)}) \to (Val(B_1,A_1),\M_{Val(B_1,A_1)})$$
is a morphism of locally ringed spaces.
\end{prop}

\begin{proof}
If $\varphi :(B_1,A_1) \to (B_2,A_2)$ is an adic homomorphism,
then by Lemma \ref{val to val} the pullback of every valuation in $Val(B_2,A_2)$ is already in $Val(B_1,A_1)$.
Hence $\varphi_{bir}$ sends $v \in Val(B_2,A_2)$ to $v\circ \varphi \in Val(B_1,A_1)$.
The induced homomorphism of stalks is the canonical map 
\begin{displaymath}
	\xymatrix{
	{B_1}_{\varphi^{-1}(\p)} \ar[r] & {B_2}_{\p} \\
	S_{v\circ \varphi} \ar[u] \ar[r] & S_v \ar[u]
	}
\end{displaymath}
and both the top and bottom arrows are local homomorphisms.

For the  opposite direction, denote $\X_1=Val(B_1,A_1)$ and $\X_2=Val(B_2,A_2)$ and assume that the morphism of ringed spaces $(\varphi_{bir},\varphi_{bir}^\sharp):(\X_2,\M_{\X_2}) \to (\X_1,\M_{\X_1})$ is a morphism of locally ringed spaces.
This is the same as saying that the pull back morphism $\varphi^*:Spa(B_2,A_2) \to Spa(B_1,A_1)$ restricts to a morphism $\X_2 \to \X_1$.
We want to show that $B_1 \otimes_{A_1}A_2 \to B_2$ is integral.

Let $b \in B_2$.
We want to show that $b$ is integral over $B_1 \otimes_{A_1}A_2$.
If $b \in nil(B_2)$ there is nothing to prove.
Else, there is a prime $\p$ of $B_2$ such that $b \notin \p$.
So there is some $v \in \X_2$ (with kernel $\p$) such that $v(b) \neq 0$.
By assumption, the $A_1$-valuation $\varphi^*(v)=v \circ \varphi$ on $B_1$ is not bounded so there exists some $b' \in B_1$ such that $v(b) \le v \circ \varphi(b') $.

Denote by $\tilde{b}$ the image of $b$ in the localization $(B_2)_{\varphi(b')}$, and by $\tilde{b}'$ the image of $b'$ in  $(B_1)_{b'}$.
For the induced valuation on  $(B_2)_{\varphi(b')}$ we have $0 < v(\tilde{b}) \le v \circ \varphi(\tilde{b}') $.
It follows that for $b''=\frac{\tilde{b}}{\varphi(\tilde{b}')} \in (B_2)_{\varphi(b')}$ we have $v(b'') \le 1$. 

Denote the image of $A_1$ in $(B_1)_{b'}$ by $A'_1$ and the image of $A_2$ in $(B_2)_{\varphi(b')}$ by $A'_2$.
It is enough to show that $b''$ is integral over $(B_1)_{b'} \otimes_{A'_1\left[ \frac{1}{b'}\right] }A'_2\left[ \frac{1}{\varphi(\tilde{b}')}\right] $.
Hence we may assume that $b \in B_2$, $b \notin nil(B_2)$ and there is some $v \in \X_2$ with $\p=ker(v)$ such that $0 < v(b) \le 1$, and show that $b$ is integral over $B_1 \otimes_{A_1}A_2$.
It is enough to show that $b$ is integral over $A_2$.

Denote the image of $b$ in the localization $(B_2)_{\p}$ by $a$.
As $0 < v(b) \le 1$, $a$ is actually is the semi-valuation ring $S_v$ of $v$.
By Proposition \ref{stalks} there are elements $g,f_1,\ldots,f_r$ in $B_2$ generating the unit ideal and a rational domain $U=\X_2\left(  \{ f_1,\ldots,f_r \}  \slash g\right)$ such that
\begin{displaymath}
b'=b|_{U} \in \OO_{\X_2} \left(U \right) = Nor_{(B_2)_g}\left( A'_2\left[ \frac{f_1}{g},\ldots,\frac{f_r}{g}\right] \right) 
\end{displaymath}
that maps to $a$ via the canonical map $\OO_{\X_2} \left(U \right) \to \OO_{\X_2,v}=S_v$, where $A'_2$ is the image of $A_2$ in $(B_2)_g$.
Again, by replacing $B_2$ with $(B_2)_g$ and $A_2$ with $A'_2$, we may assume that 
\begin{displaymath}
b \in \OO_{\X_2} \left(\X_2\left(  \{ f_1,\ldots,f_r \}  \slash 1\right) \right) = Nor_{B_2}\left( A_2\left[ f_1,\ldots,f_r\right] \right) 
\end{displaymath}
That is, $b$ is in $B_2$ and integral over  $A_2\left[ f_1,\ldots,f_r\right]$.
So there are $c_0,\ldots,c_s \in A_2\left[f_1,\ldots,f_r\right]$ such that $(b)^{s+1}+c_s(b)^{s}+\ldots+c_0=0$.
Let $c_{s+1}=1$ and set $M=\sum_{k=0}^s\sum_{j=0}^{s+1}A_2 c_j b^{k}$.
Now, $ M $ is a finite $A_2$-module contained in $B_2$.
We have $ bM \subset M $ by the integral relation.
As $ 1\in M$, we see that $Ann_{B_2}(M)=0$.
Thus $b$ is integral \cite[Chapter V \S1]{ZSI}.
\end{proof}

\begin{definition}
Given a birational space $\X$ and a pair of schemes $(Y,X)$.
If $(Y,X)_{bir}=\X$ we say that $(Y,X)$ is a \emph{scheme model} of $\X$.
Given another scheme model $(Y',X')$ of $\X$, if there is a morphism of pairs of schemes $g:(Y',X') \to (Y,X)$ such that $g_{bir}$ is the identity we say that $(Y',X')$ \emph{dominates} $(Y,X)$.
\end{definition}

\section{Relative Blow Ups}
\subsection{Construction of a Relative Blow Up}
Let $(Y \stackrel{f}{\to} X)$ be a pair of schemes.
Then $f_{*}\OO_Y$ is a quasi-coherent $\OO_X$-algebra via $f^{\sharp}:\OO_X \to f_{*}\OO_Y$.
The homomorphism $f^{\sharp}$ is an injection since $f$ is schematically dominant.
Let $\E$ be a finite quasi-coherent $\OO_X$-module contained in $f_{*}\OO_Y$ and containing $f^{\sharp}(\OO_X)$.
Then $f^{\sharp}$ induces a homomorphism of $\OO_X$-modules  $\OO_X \to \E$.
Pulling back to $Y$ we obtain a homomorphism of $\OO_Y$-modules  $\OO_Y \to f^*(\E) $.
Assume that we have an isomorphism of $\OO_Y$-modules $\varepsilon:\OO_Y \simeq f^*(\E)$ (not the homomorphism induced by $f^{\sharp}$ described above).
Using the multiplication of the $\OO_X$-algebra $f_{*}\OO_Y$ we define the product $\OO_X$-module $\E^d$ as the image of $\E^{\otimes d}$ under $f_{*}\OO_Y^{\otimes d} \to f_{*}\OO_Y$, with the tensor over $\OO_X$.
Now, the graded $\OO_X$-module $\E'=\oplus_{d \ge 0}\E^d$ (taking $\E^0=\OO_X$) is quasi-coherent and has a structure of a graded $\OO_X$-	 algebra.
We set $X_{\E}=\mathbf{Proj}_X \left(  \E'\right) $.
The construction gives a natural morphism $\pi_{\E}:X_{\E} \to X$ which is projective.
We also obtain an injection $ \E' \to f_{*} \OO_Y $ which gives rise to a natural morphism of $X$-schemes
\begin{displaymath}
 Y \simeq \mathbf{Spec}_X \left( f_{*}\OO_Y \right) \to \mathbf{Spec}_X \left(\E'\right),
\end{displaymath}
which is affine and schematically dominant.

\noindent Let $s$ be the homogeneous element of degree 1 in $\E'$ corresponding to $1 \in \OO_{X}(X)$ i.e. $s=f^{\sharp}(1) \in \E(X)$.
We obtain $\left( X_{\E}\right) _s =\mathbf{Spec}_X \left( \E' \slash  (s-1)\E' \right)$  and an affine, schematically dominant open immersion $\left( X_{\E}\right) _s \to  X_{\E}$ \cite[II\S3]{ega}.
Considered as an element of $f_{*}\OO_Y(X)$ , $s=1$.
So  $\left( X_{\E}\right) _s =\mathbf{ Spec}_X \left( \E' \right)$.
Composing we get an affine dominant morphism $f_{\E}:Y \to X_{\E}$.
In other words we have constructed a pair of schemes $(Y \stackrel{f_{\E}}{\to} X_{\E})$ with a morphism of pairs of schemes
\begin{displaymath}
g_{\E}=(id_Y,\pi_{\E}):(Y,X_{\E}) \to (Y,X).
\end{displaymath} 

\begin{definition}
We call the pair $(Y,X_{\E})$ together with the morphism $g_{\E}$ constructed above the \emph{relative blow up} of $(Y,X)$ with respect to $\E$.
\end{definition}

\begin{lemma} \label{bu-bir}
Let $X,Y,\E$ be as above. 
Then $(Y,X_{\E})_{bir}=(Y,X)_{bir}$.
\end{lemma}

\begin{proof}
The question is local on $X$ so we assume that $X=SpecA$, $Y=SpecB$, $A \subset B$ and $E=\sum_{i=0}^n Ab_i$ a finite $A$-module contained in $B$ and containing $A$.
Denote $E'=\oplus_{d \ge 0}E^d$ and $\X=Val(B,A)=(Y,X)_{bir}$.
The relative blow up is $Y \to X_E=ProjE'$.
By functoriality of $bir$ we have a continuous map $(Y,X_{E})_{bir} \to \X$.
Clearly $X_E$ is proper over $X$ so by the valuative criterion of properness  we are done.
\end{proof}

\subsection{Properties of Relative Blow Ups}
Continuing the above discussion, let $(Y' \stackrel{f'}{\to} X')$ be another pair and $h:(Y',X') \to (Y,X)$ a morphism of pairs i.e.
\begin{displaymath}
	\xymatrix{
	Y' \ar[r]^{h_Y} \ar[d]^{f'} & Y \ar[d]^f \\
	X' \ar[r]^{h_X} & X.
	}
\end{displaymath}
Then $h_{X}^{*} \left( \E \right) $ is a finite quasi-coherent $h_{X}^{*}\left( \OO_X \right)=\OO_{X'}$-module.
The inclusion $\E \subset f_* \OO_Y$ induces a homomorphism  $h_{X}^{*}\left( \E \right) \to h_{X}^{*} \left(f_* \OO_Y\right)$ of sheaves on $X'$.
The morphism of schemes $h_Y:Y' \to Y$ is equipped with a homomorphism of sheaves $\OO_Y \to h_{Y *}\OO_{Y'}$ on $Y$.
Pushing forward to $X$ and then pulling back to $X'$ we get a homomorphism of sheaves on $X'$,  $h^* _X f_* \OO_Y \to h^* _X (f \circ h_Y)_* \OO_{Y'}$.
As $(f \circ h_Y)_* \OO_{Y'} = (h_{X} \circ f')_* \OO_{Y'}$ as sheaves on $X$, we obtain a homomorphism of sheaves on $X'$  $h^{*}_{X}(f \circ h_Y)_* \OO_Y \to( f')_* \OO_Y$ (corresponding to the identity homomorphism of the sheaves on $X$ by the bijection $Hom_{X'}(h_X^*\G,\F) \simeq Hom_X(\G,{h_X}_*,\F)$).
Composing we obtain a homomorphism of sheaves on $X'$
\begin{displaymath}
h_{X}^{*} \left( \E \right) \to h^* _X f_* \OO_Y \to h^* _X (f \circ h_Y)_* \OO_{Y'} \to f'_* \OO_{Y'}.
\end{displaymath}

\begin{definition}
We call the image of the above morphism the \emph{inverse image module} of $\E$ (with respect to the morphism of pairs $h$) and denote it  $h^{-1}(\E)$.
\end{definition}

It is a finite quasi-coherent $\OO_{X'}$-module contained in $f'_* \OO_{Y'}$.

\noindent We also have
\begin{displaymath}
\OO_X \stackrel{f^{\sharp}}{\to} f^{\sharp}(\OO_X) \subset \E \subset f_* \OO_Y.
\end{displaymath}
Applying $h^* _{X}$ we get homomorphisms of $\OO_{X'}$-modules
\begin{displaymath}
h^* _{X}\OO_X \stackrel{f^{\sharp}}{\to} h^* _{X}f^{\sharp}(\OO_X) \to h^* _{X}\E \to h^* _{X}f_* \OO_Y.
\end{displaymath}
Using the isomorphism $h_{X}^{*}\left( \OO_X\right) \stackrel{\sim}{\to} \OO_{X'}$ we obtain
\begin{displaymath}
\xymatrix{
	h^* _{X}\OO_X \ar[dd] \ar[r] &  h^* _{X}f^{\sharp}(\OO_X) \ar[r] & h^* _{X}\E \ar[r] \ar[d] &  h^* _{X}f_* \OO_Y \ar[dd] \\
	& & h^{-1}\left( \E \right) \ar[dr] & \\
	\OO_{X'} \ar[urr] \ar[rrr] &  &  &  f'_{ *} \OO_{Y'},
	}
\end{displaymath}
that is, the homomorphism $\OO_{X'} \to f'_*\OO_{Y'}$ factors through $h^{-1}(\E)$.
As $f':Y' \to X'$ is schematically dominant, $h^{-1}\left( \E \right) $ contains the image of $\OO_{X'}$.

\begin{prop} \label{invert}
Let $X,Y,\E$ be as above and $g_{\E}:(Y,X_{\E}) \to (Y,X) $ the relative blow up.
Then the inverse image module $g^{-1}_{\E}\left( \E \right)$ is an invertible sheaf on $X_{\E}$.
\end{prop}

\begin{proof}

We have $g_{\E}:(Y,X_{\E}) \to (Y,X)$
\begin{displaymath}
	\xymatrix{
	Y \ar[r]^{id} \ar[d]^{f_{\E}} & Y \ar[d]^f \\
	X_{\E} \ar[r]^{\pi_{\E}} & X.
	}
\end{displaymath}
The inverse image module $g^{-1}_{\E}\left( \E \right)$ is the image of $\pi_{\E}^*(\E)$ in $f_{\E *}\OO_Y$ (substituting $h_Y=id_Y$ and $h_X=\pi_{\E}$ in the discussion above).
Again the question is local on $X$ so we assume that $X=SpecA$, $Y=SpecB$, $A \subset B$ and $E=\sum_{i=0}^n Ab_i$ such that $b_0,\ldots,b_n \in B$ generate the unit ideal.\\
Denote $A_i=\varphi_i(A)[\{\frac{b_j}{b_i}\}_{j}]\subset B_{b_i}$ ($\varphi_i$ the canonical homomorphism $B\to B_{b_i}$).
Then $X_{\E}=Proj\left( \oplus_{d \ge 0}E^d \right)$.
As we saw in Theorem \ref{sheaves} the affine charts are $X_{\E,i}=Spec( A_i ) $.
Denote by $E_i$ the image of $E \otimes_A A_i$ in  $B_{b_i}$ under the map induced by multiplication.
Since $\frac{b_j}{b_i} \in A_i$ we see that $b_j=\frac{b_j}{b_i} \cdot b_i \in E_i$.
So $E_i$ is generated by the single element $b_i$ over $A_i$.
In other words, multiplication by $b_i$ gives an isomorphism of modules $\OO_{X_{\E}}|_{X_{\E,i}} \stackrel {\sim}{\to} g^{-1}_{\E}\left( \E \right) |_{X_{\E,i}} $.
\end{proof}

\begin{prop}[Universal property] \label{univ bu}
Let $(Y,X)$ be a pair of schemes with $\E$ as above.
Let $(Y',X')$ be another pair with a morphism $h=(h_Y,h_X): (Y',X') \to (Y,X)$.
If  $\mathcal{L}=h^{-1}\left( \E \right) $ is invertible on $X'$ then $h$ factors uniquely through $g_{\E}$.
\end{prop}

\begin{proof}
As this is a local question, assume $X=SpecA$ and $Y=SpecB$ with $A \subset B$.
Then $\E(X)=E=\sum_{i=0}^n Ab_i$ and $b_0,\ldots,b_n \in B$ generate the unit ideal.
The graded homomorphism
\begin{displaymath}
\delta^{\sharp}:A[T_0,\ldots,T_n] \to \oplus_{d \ge 0} E^d \qquad
T_i \mapsto b_i \in E
\end{displaymath}
gives rise to a closed immersion $\delta: X_{\E} \to \mathbb{P}^n_A$
with $\delta^* \OO(1) =\pi^*_{\E}(\E)$, with $\pi_{\E}:X_{\E} \to X$ as in Proposition \ref{invert}.
Denote by $s_{0},\ldots,s_{n}$ the global sections in $\Gamma(X',f'_*\OO_{Y'})$ corresponding to $b_0,\ldots,b_n$ via
\begin{displaymath}
h_{X}^{*} \left( \E \right) \to h^* _X f_* \OO_Y \to h^* _X (f \circ h_Y)_* \OO_{Y'} \to f'_* \OO_{Y'}.
\end{displaymath}
Then $s_{0},\ldots,s_{n}$ generate the invertible $\OO_{X'}$-module $\mathcal{L}$.
Hence they induce a unique morphism $\psi:X' \to \mathbb{P}^n_A$ with $\mathcal{L} \simeq \psi^* \OO(1)$.
If $F \in A[T_0,\ldots,T_n]$ is a homogeneous polynomial of degree d such that $F(b_0,\ldots,b_n)=0 \in B$ (i.e. $F \in \ker{\delta^{\sharp}}$) then $F(s_{0},\ldots,s_{n})=0 \in \Gamma(X',\mathcal{L}^{\otimes d})$.
So $\psi$ factors through $X_{\E}$ and $h_X=\pi_{\E} \circ \psi'$ where $\psi':X' \to X_{\E}$ is given by
\begin{displaymath}
	\xymatrix{
	X'  \ar[r]^{\psi} \ar@{-->}[d]_{\psi'} & {\mathbb{P}^n_A} \ar[d] \\
	 X_{\E} \ar[r]^{\pi_{\E}} \ar[ur]^{\delta} & X.
	}
\end{displaymath}

Taking $h'=(h_Y,\psi')$ we factor $h$ as 
\begin{displaymath}
	\xymatrix{
(Y',X')\ar[dr]^h \ar@{-->}
[r]^{h'} & (Y,X_{E}) \ar[d]^{g_{E}} \\
  &  (Y,X)
  }
\end{displaymath}
which shows the existence.

Now, we have
\begin{displaymath}
\mathcal{L}=h^{-1}\left( \E \right)=(g_{\E} \circ h')^{-1}\left( \E \right) = h'^{-1} \left( \OO_{X_{\E}}(1) \right).
\end{displaymath}
So we have a surjective homomorphism of $X'$-modules 
\begin{displaymath}
{\psi'}^* \OO_{X_{\E}}(1)={h'_X}^* \OO_{X_{\E}}(1) \to h'^{-1} \left( \OO_{X_{\E}}(1) \right)=\mathcal{L}.
\end{displaymath}
Since both ${\psi'}^* \OO_{X_{\E}}(1)$ and $\mathcal{L}$ are invertible this homomorphism is an isomorphism, and we conclude that  $h'$ is unique.
\end{proof}

\begin{lemma} \label{bu+nor}
Let $(Y,X)$ be a pair of schemes with $\E$ as above.
Then there is a natural isomorphism of pairs of schemes
$$ (Y,Nor_Y(X_{\E})) \simeq (Y,Nor_Y(X)_{Nor_Y(\E)}). $$
\end{lemma}

\begin{proof}
The general case immediately follows from the affine case.
The affine case follows from the facts that $ S^{-1}Nor_BA=Nor_{S^{-1}B}(S^{-1}A) $ and $ (Nor_BA)[t]= Nor_{B[t]}(A[t]) $, where $(B,A)$ is a pair of rings and $S \subset A$ is a multiplicatively closed subset.
\end{proof}

Assume two finite, quasi-coherent $\OO_X$-modules $\E'$ and $\E''$ on the pair of schemes $(Y,X)$, contained in $f_{*}\OO_Y$ and containing the image of $\OO_X$, together with isomorphisms of $\OO_Y$-modules $\varepsilon':\OO_Y \simeq f^*(\E')$ and $\varepsilon'':\OO_Y \simeq f^*(\E'')$.
Then $\E=\E' \cdot \E''$ is also a finite, quasi-coherent $\OO_X$-module contained in $f_{*}\OO_Y$ and containing the image of $\OO_X$.
Note that $\E$ is the image of $\E'\otimes_{\OO_X}\E''$ under $f_{*}\OO_Y \otimes_{\OO_X} f_{*}\OO_Y \to f_{*}\OO_Y$.
Taking $\varepsilon$ to be the isomorphism obtained by the composition of isomorphisms
\begin{displaymath}
\xymatrix{
\OO_Y \simeq \OO_Y \otimes_{\OO_Y} \OO_Y \ar[r]^{\varepsilon' \otimes \varepsilon''} & f^*(\E') \otimes_{\OO_Y} f^*(\E'') \ar[r] & f^*(\E),
}
\end{displaymath}
we can form the relative blow up $(Y,X_{\E})$.
By Proposition \ref{invert}, the inverse image module $g^{-1}_{\E}\left( \E \right)=g^{-1}_{\E}\left( \E' \right)\cdot g^{-1}_{\E}\left( \E'' \right)$ is an invertible sheaf on $X_{\E}$.
Hence $g^{-1}_{\E}\left( \E' \right)$ and $g^{-1}_{\E}\left( \E'' \right)$ are also invertible on $X_{\E}$.
By the universal property of the relative blow up $(Y,X_{\E}) \to (Y,X)$ factors through both $(Y,X_{\E'}) \to (Y,X)$ and $(Y,X_{\E''}) \to (Y,X)$.

\begin{lemma}
Let $(Y,X)$ be a pair of schemes.
Assume that $X'$ is an open subscheme of $X$.
Denote $Y'=f^{-1}(X')$.
Then a relative blow up of $(Y',X')$ extends to a relative blow up of $(Y,X)$.
\end{lemma}

\begin{proof}
\cite[Corollary 3.4.4]{tem} 
\end{proof}

From the Lemma and the paragraph above it we obtain:

\begin{cor} \label{BL 2.6}
Let $(Y,X)$ be a pair of schemes with open sub-pairs of schemes $(Y_1,X_1),\ldots ,(Y_n,X_n)$.
For each $i=1,\ldots,n$ let $(Y_i,{X_i}_{\E_i}) \to (Y_i,X_i)$ be relative blow up.
Then:
\begin{enumerate}
	\item each relative blow up $(Y_i,{X_i}_{\E_i}) \to (Y_i,X_i)$ extends to a relative blow up $(Y,X_{\E'_i}) \to (Y,X)$.
	\item there is a relative blow up $(Y,X_{\E}) \to (Y,X)$ which factors through each $(Y,X_{\E'_i}) \to (Y,X)$.	 
\end{enumerate}
\end{cor}

\section{Birationl Spaces in Terms of Pairs of Schemes}

In this section we show that the \textit{bir} functor provides an equivalence of categories between the localization of the category of pairs of schemes, with respect to the class of relative blow-ups and relative normalizations, and the category of quasi-compact and quasi-separated birational spaces.

First we need to show that the functor \textit{bir} takes relative blow-ups and relative normalizations to isomorphisms.
Lemma \ref{bu-bir} gives the result for relative blow ups.
As for relative normalizations, given a pair of schemes $(Y,X)$, note that for every affine $U \subset X$ we have $\left( Nor_Y \OO_X \right)  (U)=Nor_{f_* \OO_Y(U)} \OO_X(U)$.
From Lemma \ref{nor-bir} we get the result.

\subsection{Faithfulness}

\begin{theorem}[\emph{bir} is Faithful] \label{Faithful}
Let $ (Y \stackrel{f}{\to} X) $ and $ (Y' \stackrel{f'}{\to} X') $ be pairs of schemes such that $X=Nor_YX$ and $X'=Nor_{Y'}X'$.
Denote $\X=(Y,X)_{bir}$ and $\X'=(Y',X')_{bir}$.
Let $g_1,g_2:(Y,X) \to (Y',X')$ be morphisms of pairs of schemes.
If $g_{1,bir}=g_{2,bir}$ as morphisms of the birational spaces $\X \to \X'$.
Then $g_1=g_2$.
\end{theorem}

\begin{proof}
Denote $g_{1,bir}=g_{2,bir}=h$.

We have a commutative diagram
\begin{displaymath} 
  \xymatrix { 
  Y \ar[d]^{\sigma} \ar@<0.5ex>[r]^{g_{1,Y}} \ar@<-0.5ex>[r]_{g_{2,Y}} & Y' \ar[d]^{\sigma}    \\
  {\X} \ar[d]^{\tau} \ar[r]^h  & {\X'}\ar[d]^{\tau} \\
  X \ar@<0.5ex>[r]^{g_{1,X}} \ar@<-0.5ex>[r]_{g_{2,X}} & X' .   
         }
\end{displaymath}
It follows from Proposition \ref{sigtau} and Lemma \ref{sigtau2} that $g_1$ and $g_2$ agree on the underlying topological spaces (both $Y$ and $X$).
Denote the topological part of  $g_1$ and $g_2$ by $g=(g_Y,g_X)$.
Furthermore, as in the faithfulness part of Theorem \ref{important theorem},  $g_{1,Y}$ and $g_{2,Y}$ agree as morphisms of schemes.

Let  $\{ (Y'_i,X'_i) \}_{i \in I}$ be an open affine covering of $(Y',X')$. 
Denote the topological pull back of each $(Y'_i,X'_i)$ through $g$ by $(Y_i,X_i)$.
For every ${i \in I}$ there is an open affine covering $\{ (Y'_{ij},X'_{ij}) \}_{j \in J_i}$ of $(Y'_i,X'_i)$.
It is enough to show that the restrictions $g_{1,X}|_{X_{ij}}$ and $g_{2,X}|_{X_{ij}}$ agree as morphisms of schemes $X_{ij} \to X'_i$ for each $j \in J_i$ and $i \in I$.
Hence we may assume that $(Y',X')$ and $(Y,X)$ are affine pairs.
This was already proved in the faithfulness part of Theorem \ref{important theorem}.

\end{proof}

\subsection{Fullness}

By combining \cite[Lemma 3.4.6]{tem} and our Remark \ref{tau2} we obtain

\begin{lemma}
Given a quasi-compact open subspace $\mathfrak{U} \subset \X=Val(Y,X)$, there exists a relative blow up $(Y,X_{\E}) \to (Y,X)$ and an open subscheme $U$ of $X_{\E}$ such that $\mathfrak{U} = \tau^{-1}(U)=Val(f^{-1}_{\E}(U),U)$.
\end{lemma}

Taken together with Corollary \ref{BL 2.6}, we immediately obtain

\begin{cor} \label{BL 4.4}
Let $(Y,X)$ be a pair of schemes and let $\Omega$ be a finite family of quasi-compact open subspaces of the associated birational space $(Y,X)_{bir}$.
Then there is relative blow up $(Y,X_{\E}) \to (Y,X)$ together with a family $\Omega'$ of open subschemes of $X_{\E}$ such that the associated family $\Omega'_{bir}$ coincides with $\Omega$.
Furthermore if $\Omega$ covers $(Y,X)_{bir}$, the family $\Omega'$ covers $X_{\E}$.
\end{cor}

We are now ready to prove fullness.

\begin{theorem}[\emph{bir} is Full] \label{Full}
Let $(Y,X)$ and $(Y',X')$ be pairs of schemes and let $h: (Y',X')_{bir} \to (Y,X)_{bir}$ be a morphism of birational spaces.
Then there exist a relative blow up $(Y',X'_{\E}) \to (Y',X')$ and a morphism of pairs $k: (Y',Nor_{Y'}X'_{\E}) \to (Y,X)$ such that $k_{bir} = h \circ g_{bir}$, where $g$ is the morphism $ (Y',Nor_{Y'}X'_{\E}) \to (Y',X'_{\E}) \to (Y',X') $.\\
In particular $h$ is isomorphic to $k_{bir}$.
\end{theorem}

\begin{proof}
As \emph{bir} factors through the localized category, by Lemma \ref{bu+nor} we may replace $(Y',X')$ with $(Y',Nor_{Y'}X')$.
So the morphism $g$ of the statement is just the relative blow up $(Y',X'_{\E}) \to (Y',X')$.

Consider the affine case.
Let $(B,A)$ and $(B',A')$ be two pairs of rings.
By Theorem \ref{important theorem} the morphism between the associated biratonal spaces $h :(B',A')_{bir} \to (B,A)_{bir}$ is given by a morphism of the pairs of rings $(B,A) \to (B',Nor_{B'}A')=(B',A')$.
We see that the required relative blow up is just the identity.

For the general case, denote $\X=(Y,X)_{bir}$ and $\X'=(Y',X')_{bir}$.
An affine covering $ \{(Y_i,X_i) \}$ of $(Y,X)$ gives an affinoid covering $\{ \X_i \}$  of $\X$.
Each preimage $h^{-1}(\X_i)$ can also be covered by finitely many open affinoid biratonal subspaces.
Hence by Corollary \ref{BL 4.4}, refining the coverings in a suitable way and replacing $(Y',X')$ with a suitable relative blow up, we may assume that we have coverings $\{ \X_i \}$ of $\X$ and $\{ \X'_i \}$ of $\X'$ consisting of finitely many open affinoid biratonal subspaces such that $h(\X'_i) \subset \X_i$ for all $i$ and both are represented by affine open coverings $ \{(Y_i,X_i) \}$ of $(Y,X)$ and $\{ (Y'_i,X'_i) \}$ of $(Y',X')$.

By the affine case we obtain for every $i$ a relative blow-up $g_i:(Y'_i,{X'_i}_{\E_i}) \to (Y'_i,X'_i)$ and morphism $k_i:(Y'_i,{X'_i}_{\E_i}) \to (Y_i,X_i)$ satisfying $k_{i,bir} = h|_{\X'_i} \circ g_{i,bir}$.
By Corollary \ref{BL 2.6}, there is some relative blow-up $g:(Y',X'_{\E}) \to (Y',X')$ such that all the pairs $(Y'_i,{X'_i}_{\E_i})$ are open sub-pairs of $(Y',X'_{\E})$ and form a covering.
It follows from Theorem \ref{Faithful} that we can glue the $k_i$ and get a morphism $k: (Y',X'_{\E}) \to (Y,X)$  such that $k_{bir} = h \circ g_{bir}$.
\end{proof}

\begin{cor} \label{BL 4.1(d)}
Let the pairs of schemes $(Y,X)$ and $(Y',X')$ both be scheme models for the same birational space $\X$.
Then there is another scheme model $(Y'',X'')$ of $\X$ that dominates both via relative blow ups and perhaps a relative normalization.
\end{cor}

\begin{proof}
Again we replace $(Y,X)$ and $(Y',X')$ with $(Y,Nor_{Y}X)$ and $(Y',Nor_{Y'}X')$ respectively.

We have an isomorphism $h:(Y',X')_{bir} \stackrel{\sim}{\to} (Y,X)_{bir}$.
As $Y$ is embedded in the subset of $\X=(Y,X)_{bir}$ of points $v$ such that $\M_{\X,v}=\OO_{\X,v}$, its scheme structure is determined by $(\X,\M_{\X})$.
The same is true for $Y'$, so $h$ induces an isomorphism $Y \simeq Y'$.
We assume that $Y'=Y$.

Applying Theorem \ref{Full} we obtain a relative blow up $g:(Y,X'_{\E}) \to (Y',X')$ and a morphism of pairs $k: (Y',X'_{\E}) \to (Y,X)$  such that $k_{bir} = h \circ g_{bir}$, in particular  $k_{bir}: (Y',X'_{\E})_{bir} \to (Y,X)_{bir}$ is also an isomorphism.
Using Theorem \ref{Full} again for $k_{bir}^{-1}$ we obtain a relative blow up $j:(Y,X_{\F}) \to (Y,X)$ and a morphism of pairs $q:(Y,X_{\F}) \to (Y',X'_{\E})$ such that $q_{bir} = k_{bir}^{-1} \circ j_{bir}$.

\begin{displaymath}
\xymatrix{
(Y,X'_{\E}) \ar[d]^g \ar[dr]^k & (Y,X_{\F}) \ar[d]^j \ar[l]_q  \\
(Y',X') &  (Y,X) 
}
\end{displaymath}
As $j=k \circ q:(Y,X_{\F}) \to (Y,X)$ is a relative blow up, so is $q:(Y,X_{\F}) \to (Y',X'_{\E})$.
Thus the composition of relative blow ups $g \circ q:(Y,X_{\F}) \to (Y',X')$ is also a relative blow up.
So $(Y,X_{\F})$ is the required scheme model.

\end{proof}

\subsection{Essential Surjectivety}

\begin{theorem}[\emph{bir} is Essentially Surjective]
Every quasi-compact and quasi-separated birational space $\X$ has a scheme model.
\end{theorem}

\begin{proof}
Consider a quasi-compact and quasi-separated  birational space $\X$.
We want to show that there is a pair of schemes $(Y,X)$ satisfying $(Y,X)_{bir} \simeq \X$.
We proceed by induction on the number of open birational spaces which cover $\X$ and have scheme models.
As $\X$ is quasi-compact, it is enough to consider only the case of an affinoid covering consisting of two subspces.

Assume that $\X$ is covered by two quasi-compact open subspaces $\mathfrak{U}_1$ and $\mathfrak{U}_2$, which admit scheme models $(V_1,U_1)$ and $(V_2,U_2)$.
Set $\mathfrak{W} =\mathfrak{U}_1\cap \mathfrak{U}_2$.
Since $\X$ is quasi-separated, an application of Corollary \ref{BL 4.4} shows that, after blowing-up, we may assume that the open immersions $\mathfrak{W}  \subset \mathfrak{U}_1$ and $\mathfrak{W}  \subset \mathfrak{U}_2$ are represented by open immersions of sub-pairs $(T',W') \subset (V_1,U_1)$ and $(T'',W'')  \subset (V_2,U_2)$.
Now, using Corollary \ref{BL 4.1(d)}, we can dominate the scheme models $(T',W')$ and $(T'',W'')$ by a third scheme model $(T,W)$ of $\mathfrak{W}$.
Using Corollary \ref{BL 2.6} we extend the corresponding blow-ups to $(V_1,U_1)$ and $(V_2,U_2)$, so we may view $(T'',W'')$ as an open sub-pair of $(V_1,U_1)$ and $(V_2,U_2)$. Gluing both along W yields the required scheme model $(Y,X)$ of $\X$.

\end{proof}

\vspace{10mm}
\bibliography{mybib}{}
\bibliographystyle{amsalpha}

\end{document}